\documentclass[11pt,a4paper,reqno]{amsart} 

\usepackage{amsmath,amssymb,amsthm,enumerate}
\usepackage[myheadings]{fullpage}
\usepackage{tikz}
\usepackage[all]{xy}
\usepackage{graphicx}
\usepackage{subcaption}
\usepackage{hyperref}
\hypersetup{
    colorlinks=true,
    linkcolor=blue,
    citecolor=blue
}

\newtheorem{thm*}{Theorem}
\newtheorem{prop*}{Proposition}
\newtheorem{theorem}{Theorem}[section]
\newtheorem{cor}[theorem]{Corollary}
\newtheorem{prop}[theorem]{Proposition}
\newtheorem{lemma}[theorem]{Lemma}

\theoremstyle{definition}
\newtheorem{example}[theorem]{Example}
\newtheorem{definition}[theorem]{Definition}
\newtheorem{rem}[theorem]{Remark}

\makeatletter
\@namedef{subjclassname@2025}{\textup{2025} Mathematics Subject Classification}
\makeatother

\begin{document}
\raggedbottom

\title[Pseudo-Anosov-like maps on the infinite ladder surface]{On a family of pseudo-Anosov-like maps \\ on the infinite ladder surface}
\author{Nikita Agarwal}
\address{(N. Agarwal) Department of Mathematics\\
	University of Delhi\\
	Delhi 110007\\
	India}
\address{(N. Agarwal) Department of Mathematics\\
Indian Institute of Science Education and Research Bhopal\\
Bhopal Bypass Road, Bhauri \\
Bhopal 462 066, Madhya Pradesh\\
India}
\email{nagarwal@maths.du.ac.in}
\urladdr{https://sites.google.com/view/nikita-agarwal-du}

\author{Rohan Suresh Mahure}
\address{(R. S. Mahure) Department of Mathematics\\
Indian Institute of Science Education and Research Bhopal\\
Bhopal Bypass Road, Bhauri \\
Bhopal 462 066, Madhya Pradesh\\
India}
\email{rohansmahure@gmail.com, mahure22@iiserb.ac.in}
\urladdr{https://sites.google.com/iiserb.ac.in/rohanmahure/home}

\author{Kashyap Rajeevsarathy}
\address{(K. Rajeevsarathy) Department of Mathematics\\
Indian Institute of Science Education and Research Bhopal\\
Bhopal Bypass Road, Bhauri \\
Bhopal 462 066, Madhya Pradesh\\
India}
\email{kashyap@iiserb.ac.in}
\urladdr{https://home.iiserb.ac.in/~kashyap/}



\begin{abstract}

Let $S_g$ be the closed surface of genus $g$, $\mathcal{L}$ be the infinite Jacob's ladder surface, and $\mathrm{Map}(S)$ denote the mapping class group of a surface $S$. Let $q_g:\mathcal{L}\to S_g$ be the regular infinite-sheeted cover with deck transformation group $\mathbb{Z}$. In this paper, we show the existence of ``pseudo-Anosov-like'' maps on $\mathcal{L}$ that arise as the lifts of Penner-type pseudo-Anosov maps on $S_g$ under the cover $q_g$. Furthermore, we establish that these lifts are topologically transitive, mixing, and support null recurrent dynamics. Moreover, we present concrete examples of infinite families of such maps on $\mathcal{L}$.

\end{abstract}

\maketitle

\section{Introduction}
A connected and oriented surface $S$ is said to be of \textit{infinite type} if its fundamental group is not finitely generated and of \textit{finite type} otherwise. Infinite-type surfaces are classified based on their genera, boundary components, and their ends, which can either be planar or non-planar~\cite{classInfTY2,classInfTY1}. Let $\mathrm{Homeo}^+(S)$ be the group of orientation-preserving self-homeomorphisms on $S$. The \textit{mapping class group} of $S$, denoted by $\mathrm{Map}(S)$, is defined as the group of isotopy classes of orientation-preserving self-homeomorphisms on $S$. The mapping class group of an infinite-type surface, also known as the \textit{big mapping class group}, is an uncountable and complete topological group~\cite[Chapter 12]{big}. One of the most significant questions in the study of big mapping class groups is whether it is possible to establish an analogue of the Nielsen-Thurston classification~\cite{flp,WT} in this context. Although some progress has been made in this area~\cite{bestvinaTowNT}, this question remains largely unresolved. Let $S_g$ denote the closed orientable surface of genus $g \geq 2$. Let $\mathcal{L}$ denote the infinite Jacob's ladder surface. Consider the infinite-sheeted cover $q_g:\mathcal{L} \to S_g$ induced by the $\mathbb{Z}$-action on $\mathcal{L}$ generated by the standard right handle shift map $h_{g-1}$ (by $g-1$ units) on $\mathcal{L}$ (see Figure~\ref{fig:infcover}). Motivated by the classification problem, in this paper, we analyze the dynamical properties of a family of maps on $\mathcal{L}$ that are lifts of Penner-type pseudo-Anosov maps on $S_g$ under $q_g$. First, we study the structure of lifts by constructing the lift homeomorphism (Section~\ref{sec:mapping-class}). Then, we turn our focus to the lifts of pseudo-Anosov maps. To study the dynamical properties of lifts of pseudo-Anosov maps, we lift Thurston's finite Markov partitions~\cite{flp} to create a countable symbolic coding (Section~\ref{subsec:symcodingsec}). Next, we translate this topological structure into the language of countable-state Markov shifts (Section~\ref{subsec:symcodingsec}) to get Thurston-style symbolic coding. Finally, by recognizing the block-tridiagonal structure of the resulting infinite matrix, we deploy quasi-birth-death process techniques to analyze recurrence (Section~\ref{sec:QBD}). 

Let $F \in \mathrm{Mod}(S_g)$ be a pseudo-Anosov mapping class that is represented by a homeomorphism $f:S_g\rightarrow S_g$ that lifts to a homeomorphism $\widetilde{f}:\mathcal{L}\rightarrow \mathcal{L}$ under $q_g$. First, we develop the requisite tools to study the properties of the dynamical system $(\mathcal{L},\widetilde{f})$ using the tools available for the well-studied dynamical system $(S_g, f)$. Thurston~\cite{flp} uses Markov partitions to show that $(S_g, f)$ is topologically transitive and mixing, with the topological entropy of $(S_g,f)$ being equal to $\ln(\lambda)$, where $\lambda$ is the stretch factor of $f$. By taking the pre-image of a Markov partition of $(S_g, f)$ under the map $q_g$, we obtain a Markov partition for $(\mathcal{L}, \widetilde{f})$, which in turn yields a symbolic coding for this dynamical system (see Section~\ref{subsec:symcodingsec}). Following the notation in the discussion above, we now state the first main result of this paper.

\begin{thm*}\label{thm1}
 For any lift $\widetilde{f}:\mathcal{L}\to\mathcal{L}$ of a pseudo-Anosov homeomorphism $f: S_g \to S_g$ under $q_g$ ($g\ge 3$), there exists a countable Markov partition for $(\mathcal{L},\widetilde{f})$. Hence, there exists a countable Markov shift $(\widetilde{\Sigma}_{\widetilde{A}}, \sigma)$, which is semi-conjugate to $(\mathcal{L}, \widetilde{f})$, that is, there exists an onto continuous map $\widetilde{\pi}: \widetilde{\Sigma}_{\widetilde{A}} \to 
        \mathcal{L}$ such that the following diagram commutes:
        \[
        \xymatrix{
            & \widetilde{\Sigma}_{\widetilde{A}} \ar_{\widetilde{\pi}}[d] \ar^\sigma[r] & \widetilde{\Sigma}_{\widetilde{A}} \ar^{\widetilde{\pi}}[d] \\
            & \mathcal{L} \ar^{\widetilde{f}}[r] & \mathcal{L}.
        }
        \]
        Also, $\widetilde{\pi}$ is a finite-to-one map.
    \end{thm*}

The matrix $\widetilde{A}$ related to the symbolic coding of $(\mathcal{L},\widetilde{f})$ is closely connected with the matrix corresponding to a symbolic coding of $(S_g,f)$. Its special structure allows us to use techniques from the theory of Quasi-Birth-Death Processes (QBD)~\cite{QBD}, and we obtain the following result.

\begin{thm*}\label{thm2}
If the matrix $\widetilde{A}$ related to the symbolic coding of $(\mathcal{L},\widetilde{f})$ is irreducible, then $\widetilde{A}$ is either transient or null recurrent.
\end{thm*}
 
Pseudo-Anosov maps on finite-type surfaces admit the structure of measurable dynamical systems. There is a unique invariant probability measure $\mu$ on $S_g$, known as the Parry measure, with respect to which the pseudo-Anosov map $f$ is measure-preserving, ergodic, and $\mu$ also maximizes the measure-theoretic entropy, that is, it satisfies the variational principle. This measure is built by assigning weights to rectangles defined via the transverse measures on the invariant foliations. In this setting, both the Poincar\'e recurrence theorem and Kac's lemma~\cite{YuriPolliDSES} are applicable, and such systems are always recurrent. The Poincaré recurrence theorem states that for a measure-preserving transformation on a finite measure space, almost every point of any measurable set returns to that set infinitely often. Building on this, Kac’s lemma quantifies the recurrence by asserting that the average first return time to a measurable set of positive measure is finite and equals the reciprocal of the measure of that set. Given a Penner-type pseudo-Anosov map on $S_g$, we show that there exists a unique lift $\widetilde{f}$ on $\mathcal{L}$ such that either $\widetilde{f} \in C_{\langle h_{g-1} \rangle}(\mathrm{Homeo}^+(\mathcal{L}))$ or $\widetilde{f}^2 \in C_{\langle h_{g-1} \rangle}(\mathrm{Homeo}^+(\mathcal{L}))$\footnote{$C_{\langle h_{g-1} \rangle}(\mathrm{Homeo}^+(\mathcal{L}))$ is the centralizer of the action of $h_{g-1}$ inside $\mathrm{Homeo}^+(\mathcal{L})$, that is, it is the collection of homeomorphisms in $\mathrm{Homeo}^+(\mathcal{L})$ that commute with the action of $h_{g-1}$}, and the matrix $\widetilde{A}$ corresponding to the symbolic coding of $(\mathcal{L},\widetilde{f})$ is irreducible (see Remark~\ref{rem:block-diag}). We will define a $\sigma$-finite measure $\widetilde{\mu}$ on $\mathcal{L}$ by lifting the measure $\mu$, with respect to which the lift $\widetilde{f}$ is measure-preserving, ergodic and conservative. As this measure is conservative, the Poincar\'e recurrence theorem is still applicable, but we will show that Kac's lemma does not hold. In particular, in Sections~\ref{subsec:TopTran}--\ref{subsec:ergnull}, we use Theorems~\ref{thm1} and~\ref{thm2} to establish the following result.
 
 \begin{thm*}\label{thm3}
For $g \geq 3$, let $f$ be a Penner-type pseudo-Anosov map on $S_g$ that factors into a product of liftable Dehn twists under $q_g$. Then there exists a unique lift $\widetilde{f}$ of $f$ on $\mathcal{L}$ such that: 
\begin{enumerate}[(i)]
\item either $\widetilde{f} \in C_{\langle h_{g-1} \rangle}(\mathrm{Homeo}^+(\mathcal{L}))$ or $\widetilde{f}^2 \in C_{\langle h_{g-1} \rangle}(\mathrm{Homeo}^+(\mathcal{L}))$,
\item the dynamical system $(\mathcal{L}, \widetilde{f})$ is topologically transitive and mixing, and 
\item the lift $\widetilde{f}$ is ergodic and null recurrent with respect to the measure $\widetilde{\mu}$.
\end{enumerate}
\end{thm*}

\noindent For more details, we refer the reader to Corollary~\ref{cor:forf2} and Theorems~\ref{thm:TopoTranandmixingth}, \ref{thm:ergth}, and \ref{thm:nullreccmapth}. An intriguing question that Theorem~\ref{thm3} raises is that of the possible existence of pseudo-Anosov-like maps on $\mathcal{L}$ that have positively recurrent dynamics. While this question is not addressed in this paper, we hope to study this phenomenon in future. Finally, in Section~\ref{sec:example}, we will provide explicit examples of families of liftable Penner-type pseudo-Anosov maps on $S_g$ and their lifts on $\mathcal{L}$. 

\section{Preliminaries}
In this section, we introduce various concepts from mapping class groups~\cite{pri} and dynamical systems~\cite{inferg,kit,YuriPolliDSES} that are central to the theory we will develop in this paper. 
\subsection{\texorpdfstring{Liftability under the cover of $S_g$ by the Jacob's ladder surface}{Liftability under the cover of Sg by the Jacob's ladder surface}}\label{subsec:liftability}
Let $S_g$ be a closed, connected, orientable surface of genus $g \geq 3$. Let $\mathcal{L}$ be the Jacob's infinite ladder surface with two non-planar ends (as shown in the top of Figure~\ref{fig:infcover}). It is well-known that $\mathrm{Map}(\mathcal{L})$ is a Polish group with respect to the topology induced by the compact-open topology of $\mathrm{Homeo}^+(\mathcal{L})$, which makes it separable and completely metrizable.

Consider a compact exhaustion $(C_n)_{n\ge 1}$ of $\mathcal{L}$, where each $C_n$ is a compact subsurface of $\mathcal{L}$, for each $n\ge 1$, $C_n\subset C_{n+1}$, and $\cup_n C_n=\mathcal{L}$. Let $C(\mathcal{L})$ denote the countable collection of free isotopy classes of simple closed curves on the surface $\mathcal{L}$. We can enumerate these classes as $C(\mathcal{L}) = \{ c_n \}_{n \ge 1}$, such that for each $n \ge 1$, the set $\{c_1, c_2, \ldots, c_n\}$ is contained in $C_n$. We define a distance function $d: \mathrm{Map}(\mathcal{L}) \times \mathrm{Map}(\mathcal{L}) \to \mathbb{R}$ as
\[
d(F, G) = \inf_{n \ge 1} \{1,\ 2^{-n} \ :\ F(c_k) = G(c_k), \text{ for all } k \leq n \}.
\]
Then the function $\rho: \mathrm{Map}(\mathcal{L}) \times \mathrm{Map}(\mathcal{L}) \to \mathbb{R}$ given by
$
\rho(F, G) = d(F, G) + d(F^{-1}, G^{-1})
$
defines a complete metric on $\mathrm{Map}(\mathcal{L})$. The metric function arising from another compact exhaustion of $\mathcal{L}$ and another enumeration of curves in the collection $C(\mathcal{L})$ gives the same topology on $\mathrm{Map}(\mathcal{L})$ as is given by the metric $\rho$.

Let $h_{g-1}$ be the standard handle shift map on $\mathcal{L}$ with the deck transformation group $\langle h_{g-1} \rangle \cong \mathbb{Z}$. Let $q_g: \mathcal{L} \to S_g$ be the regular infinite-sheeted cover induced by an $\langle h_{g-1} \rangle$ action on $\mathcal{L}$ (see Figure~\ref{fig:infcover}). Also, let $p_k: S_{g_k} \to S_g$ be the standard regular $k$-sheeted cover of $S_g$ (see Figure~\ref{fig:finite-cover} for $k=3$), induced by the $\mathbb{Z}_k$-action on $S_{g_k}$, where $g_k = k(g - 1) + 1$. 
 \noindent
\begin{figure}[ht]
    \centering
   \begin{subfigure}[b]{0.45\textwidth}
        \centering
        \begin{tikzpicture}
           
            \node[anchor=south west,inner sep=0] (image) at (0,0) {\includegraphics[width=\linewidth]{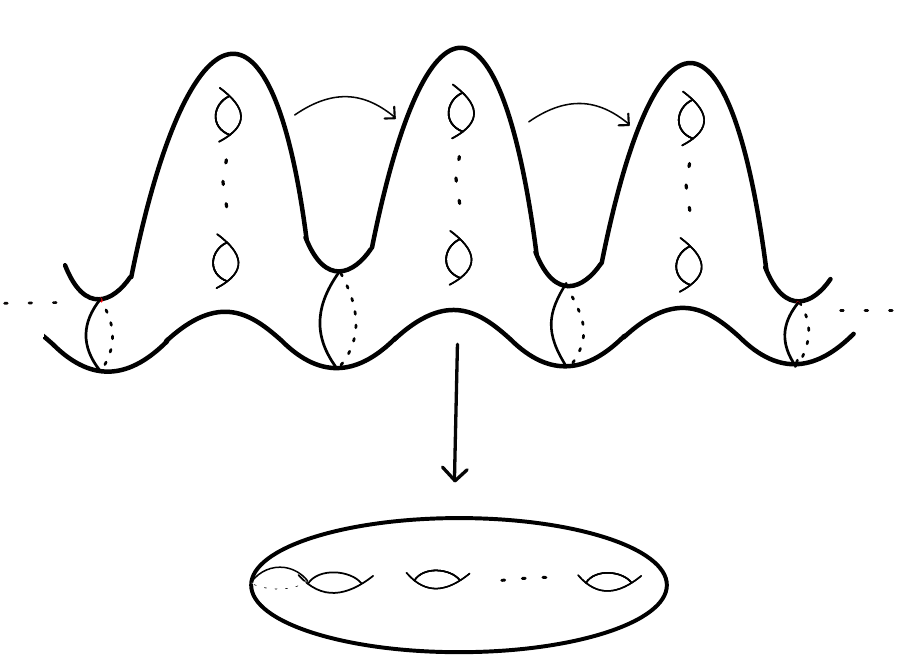}};

            \begin{scope}[x={(image.south east)},y={(image.north west)}]
                
                 \node at (0.13, 0.85) {$\mathcal{L}$};
                   \node at (0.24, 0.18) {$S_g$};
                \node at (0.55, 0.35) {$q_g$};

                \node at (0.37, 0.9) {$h_{g-1}$};
                \node at (0.65, 0.9) {$h_{g-1}$};
            \end{scope}
        \end{tikzpicture}
        \caption{The cover $q_g: \mathcal{L} \to S_g$.}
        \label{fig:infcover}
    \end{subfigure}
    \ \
   \begin{subfigure}[b]{0.45\textwidth}
        \centering
        \begin{tikzpicture}
            \node[anchor=south west,inner sep=0] (image) at (0,0) {\includegraphics[width=0.6\linewidth]{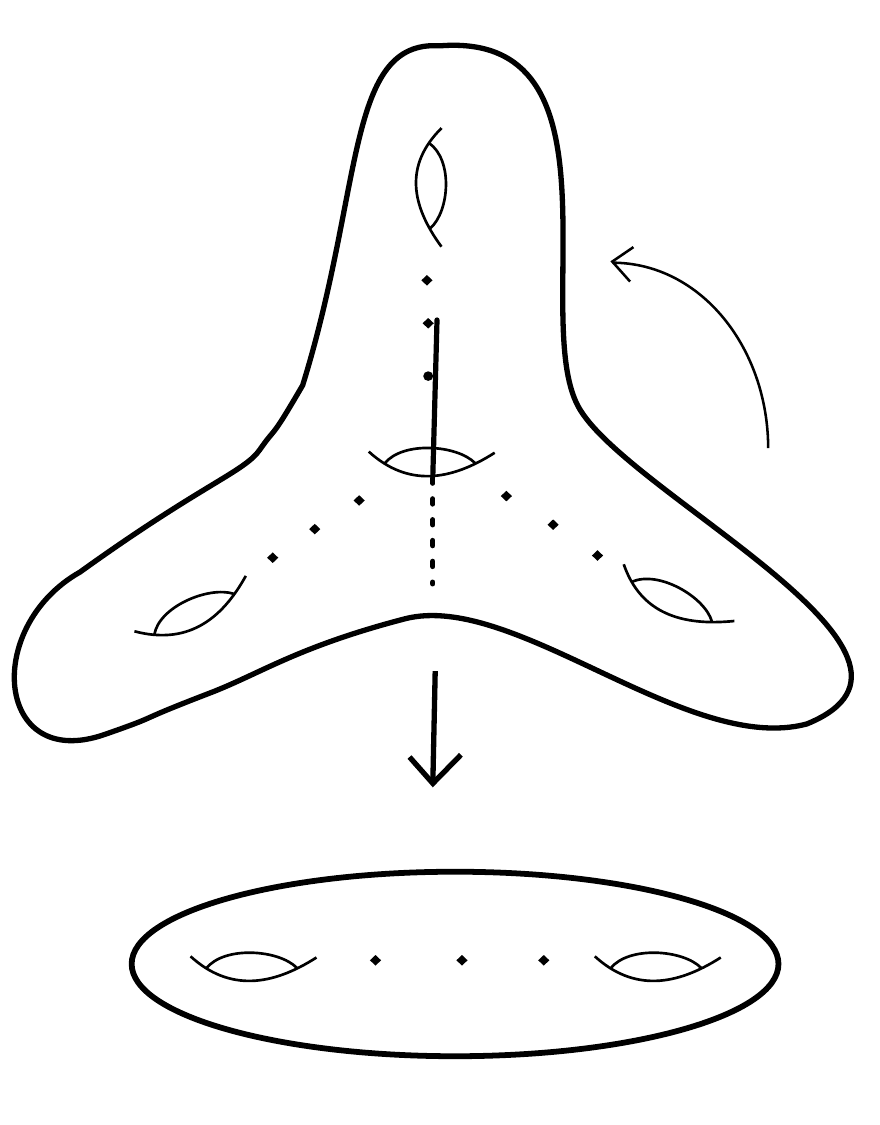}};

            \begin{scope}[x={(image.south east)},y={(image.north west)}]
               
                \node at (0.6, 0.35) {$p_3$};
                 \node at (0.25, 0.75) {$S_{g_3}$};
                \node at (0.9, 0.75) {$\frac{2\pi}{3}$};
                \node at (0.1, 0.18) {$S_g$};
            \end{scope}
        \end{tikzpicture}
        \caption{The cover $p_3: S_{g_3} \to S_g$.}
        \label{fig:finite-cover}
    \end{subfigure}
    \caption{Regular cyclic covers of $S_g$ by $\mathcal{L}$ and $S_{g_3}$.}
    \label{fig:both-covers}
\end{figure}

Let $\mathrm{LMod}_{q_g}(S_g)$ (resp. $\mathrm{LMod}_{p_k}(S_g)$) be the liftable mapping class group under the cover $q_g:\mathcal{L}\to S_g$ (resp. $p_k: S_{g_k} \to S_g$). Let $T_c$ denote the left-handed Dehn twist along a simple closed curve $c$ on $S_g$ (or $\mathcal{L}$). In order to study the lifts of the pseudo-Anosov mapping classes in $\mathrm{LMod}_{q_g}(S_g)$, we use the following explicit generating set for $\mathrm{LMod}_{q_g}(S_g)$, given by Dey~\textit{et al.}~\cite[Theorem 3]{deyGen}.
 
\begin{theorem}\label{thm:deygen}
      For $g \geq 3$, 
      \begin{enumerate}[(i)]
          \item $\mathrm{LMod}_{q_g}(S_g) = \bigcap\limits_{k \geq 2} \mathrm{LMod}_{p_k}(S_g)$, and
          \item $\textstyle{\mathrm{LMod}_{q_g}(S_g) = \langle T_{a_1}, T_{a_2}, T_{b_2}, \ldots, T_{a_g}, T_{b_g}, T_{c_1}, \ldots, T_{c_{g-1}}, T_{\alpha_1}T^{-1}_{\iota(\alpha_1)}, \ldots, T_{\alpha_{g-2}}T^{-1}_{\iota(\alpha_{g-2})}, \iota \rangle}$,      \\
          where $\iota$ represents the hyperelliptic involution. The curves are shown in Figure~\ref{fig:generators}.
     \end{enumerate}   
\end{theorem}   
      \begin{figure}[htbp]
          \centering
          \begin{subfigure}[b]{0.45\linewidth}
           \centering
          \begin{tikzpicture}
            \node[anchor=south west, inner sep=0] (image) at (0,0) {\includegraphics[width=5cm]{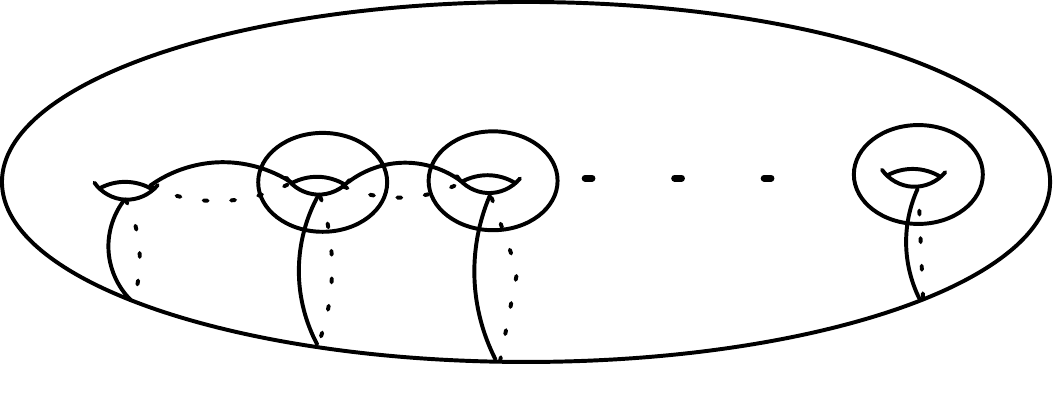}};

            \begin{scope}[x={(image.south east)}, y={(image.north west)}]
                
                \node at (0.1, 0.1) {$a_1$};
                \node at (0.3, 0.05) {$a_2$};
                \node at (0.5, 0.05) {$a_3$};
                \node at (0.9, 0.1) {$a_g$};

                \node at (0.3, 0.8) {$b_2$};
                \node at (0.49, 0.8) {$b_3$};
                \node at (0.82, 0.8) {$b_g$};

                \node at (0.2, 0.75) {$c_1$};
                \node at (0.4, 0.75) {$c_2$};
            \end{scope}
        \end{tikzpicture}
        \caption{Curves giving generators of $\mathrm{LMod}_{q_g}(S_g)$ that are Dehn twists.}
        \label{fig:dehn-curves}
    \end{subfigure}
    \hspace{0.2cm}
    \begin{subfigure}[b]{0.45\linewidth}
        \centering
        \begin{tikzpicture}
            \node[anchor=south west, inner sep=0] (image) at (0,0) {\includegraphics[width=5.5cm]{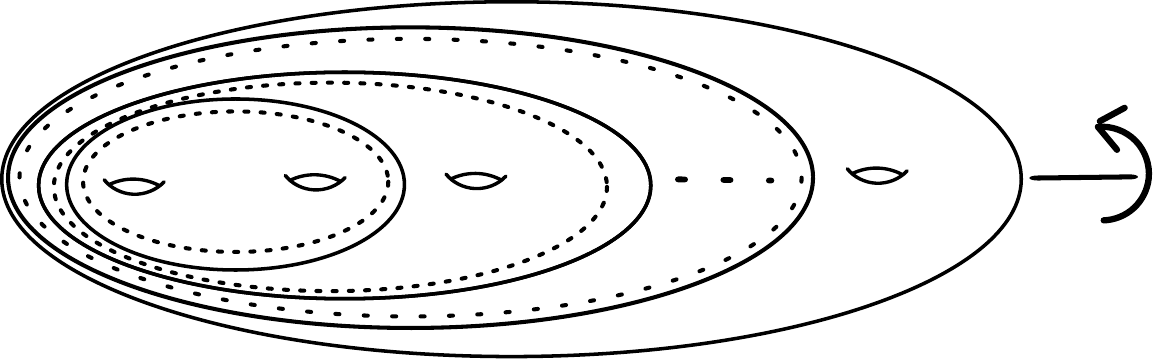}};

            \begin{scope}[x={(image.south east)}, y={(image.north west)}]
             
                \node at (0.2, 0.6) {$\alpha_1$};
                \node at (0.45, 0.6) {$\alpha_2$};
                \node at (0.78, 0.62) {$\alpha_{g-2}$};
                \node at (1.02, 0.62) {$\pi$};
                \node at (1.01, 0.3) {$\iota$};
            \end{scope}
        \end{tikzpicture}
        \caption{Curves giving generators of $\mathrm{LMod}_{q_g}(S_g)$ that are bounding pair maps and the hyperelliptic involution $\iota$.}
        \label{fig:bp-curves}
    \end{subfigure}
    \caption{Curves giving generators of $\mathrm{LMod}_{q_g}(S_g)$.}
    \label{fig:generators}
\end{figure}

We are interested in understanding the lifts of pseudo-Anosov homeomorphisms. Let $F \in \mathrm{LMod}_{q_g}(S_g)$ be a pseudo-Anosov mapping class represented by a homeomorphism $f : S_g \to S_g$ that lifts to a homeomorphism $\widetilde{f} : \mathcal{L} \to \mathcal{L}$ under the cover $q_g$. Let $(\mathcal{F}_s, \mu_s)$ and $(\mathcal{F}_u, \mu_u)$ be a transverse pair of singular arational measured foliations preserved by $f$ (see~\cite[Section 15.2.1]{pri}). Note that we call a measured foliation arational if it does not have any closed leaf. Furthermore, there exists a real number $\lambda > 1$, known as the \textit{stretch factor}, such that
\[
f((\mathcal{F}_u, \mu_u)) = (\mathcal{F}_u, \lambda \mu_u), \text{ and } f((\mathcal{F}_s, \mu_s)) = (\mathcal{F}_s, \lambda^{-1} \mu_s).
\]
Note that $(\mathcal{F}_s, \mu_s)$ (resp. $(\mathcal{F}_u, \mu_u)$) are known as stable (resp. unstable) foliations.  The following explicit construction of pseudo-Anosov mapping classes is due to Penner~\cite{penn}.
 
 \begin{theorem}\label{thm:penner}
Let $\mathcal{C} = \{\gamma_1,\ldots,\gamma_n\}$ and $\mathcal{D} = \{\gamma_{n+1},\ldots, \gamma_{n+m}\}$ be multicurves in $S_g $ such that $\mathcal{C} \cup \mathcal{D}$ fills $S_g$. 
\begin{enumerate}[(i)]
\item Then any product of positive powers of the $T_{\gamma_i}$, for $1 \leq i \leq n$, and negative powers of the $T_{\gamma_{n+j}}$, for $1 \leq j \leq m$, where each $T_{\gamma_i}$ and each $T_{\gamma_{n+j}}$ appears at least once, is pseudo-Anosov.
\item To each $T_{\gamma_k}$, associate a matrix $B_k = I +A_k\Omega$, where $I$ is the identity matrix
of order $n + m,$ $A_k$ is the $(n + m) \times (n + m)$ matrix all of whose entries are zero, except for the $(kk)^{th}$ entry, which is 1, and $\Omega$ is the incidence matrix of the embedded graph $\mathcal{C} \cup \mathcal{D}$ in $S_g$. Then, for a pseudo-Anosov $f = T_{\gamma_{i_1}}^{\delta_{i_1}}\ldots T_{\gamma_{i_s}}^{\delta_{i_s}}$ as in $(i)$, its stretch factor $\lambda$ is the largest eigenvalue of the (Perron) incidence intersection matrix $M_f = B_{i_1}^{|\delta_{i_1}|}\cdots B_{i_s}^{|\delta_{i_s}|}$.
\end{enumerate}
\end{theorem}

In view of Theorem~\ref{thm:deygen}, an $F \in \mathrm{LMod}_{q_g}(S_g)$ is represented by an $f$ that lifts to $f_k : S_{g_k} \to S_{g_k}$ under the cover $p_k$ for every $k \geq 2$. Furthermore, each $f_k$ lifts to an $\widetilde{f}$ on $\mathcal{L}$ under $q_{g_k}$. Moreover, the map $f_k$ is also pseudo-Anosov with the same stretch factor as $f$, preserving a transverse pair of singular measured foliations $(\mathcal{F}^k_s, \mu^k_s)$ and $(\mathcal{F}^k_u, \mu^k_u)$ on $S_{g_k}$ (arising as pullbacks of the foliations on $S_g$ preserved by $f$) satisfying
\[p_k((\mathcal{F}^k_s, \mu^k_s)) = (\mathcal{F}_s, \mu_s), \text{ and } p_k((\mathcal{F}^k_u, \mu^k_u)) = (\mathcal{F}_u, \mu_u).\]
Consequently, the pair of foliations on $S_g$ preserved by $f$, pull back to a pair of foliations on $\mathcal{L}$ preserved by $\widetilde{f}$. These foliations on $\mathcal{L}$ are arational with a countably infinite set of isolated singularities. We summarize these observations in the following proposition.

\begin{prop}\label{prop:foliation}
Let $F \in \mathrm{LMod}_{q_g}(S_g)$ be a pseudo-Anosov mapping class represented by a homeomorphism $f : S_g \to S_g$ preserving a transverse pair $(\mathcal{F}_s,\mu_s)$ and $(\mathcal{F}_u,\mu_u)$ of singular measured foliations with stretch factor $\lambda$. Let the homeomorphism $\widetilde{f} : \mathcal{L} \to \mathcal{L}$ be a lift of $f$ under the cover $q_g$. Then
 \begin{enumerate}[(i)]
\item There exists a transverse pair $(\widetilde{\mathcal{F}}_s,\widetilde{\mu}_s)$ and $(\widetilde{\mathcal{F}}_u,\widetilde{\mu}_u)$ of singular measured foliations on $\mathcal{L}$ that are arational with a countably infinite set of isolated singularities satisfying
\[
q_g((\widetilde{\mathcal{F}}_s,\widetilde{\mu}_s)) = (\mathcal{F}_s, \mu_s), \text{ and } q_g((\widetilde{\mathcal{F}}_u,\widetilde{\mu}_u)) = (\mathcal{F}_u, \mu_u).
\]
\item Furthermore,
\end{enumerate}
            \[\widetilde{f}((\widetilde{\mathcal{F}}_u,\widetilde{\mu}_u)) = (\widetilde{\mathcal{F}}_u,\lambda\widetilde{\mu}_u),  \text{ and }  \widetilde{f}((\widetilde{\mathcal{F}}_s,\widetilde{\mu}_s)) = (\widetilde{\mathcal{F}}_s,\lambda^{-1}\widetilde{\mu}_s).\]
\end{prop}

Note that if $(\widetilde{\mathcal{F}}_s,\widetilde{\mu}_s)$ (resp. $(\widetilde{\mathcal{F}}_u,\widetilde{\mu}_u)$) has closed leaves, then $(\mathcal{F}_s, \mu_s)$ (resp. $(\mathcal{F}_u, \mu_u)$) also has closed leaves which contradicts that $(\mathcal{F}_s, \mu_s)$ and $(\mathcal{F}_u, \mu_u)$ are arational foliations. Thus, $(\widetilde{\mathcal{F}}_s,\widetilde{\mu}_s)$ and $(\widetilde{\mathcal{F}}_u,\widetilde{\mu}_u)$ are also arational foliations. Also, $(\widetilde{\mathcal{F}}_s,\widetilde{\mu}_s)$ (resp. $(\widetilde{\mathcal{F}}_u,\widetilde{\mu}_u)$) fills the $\mathcal{L}$ as $(\mathcal{F}_s, \mu_s)$ (resp. $(\mathcal{F}_u, \mu_u)$) fills $S_g$ i.e. every essential simple closed curve has positive intersection with the foliation. From Proposition~\ref{prop:foliation}, it is apparent that the lifts of pseudo-Anosov maps on $S_g$ under $q_g$ exhibit analogous topological properties. This motivates the study of the dynamical properties of these lifts such as mixing, ergodicity and recurrence.     

\subsection{Dynamics of pseudo-Anosov maps}
\subsubsection{Symbolic Dynamics}
In~\cite{flp}, Thurston introduces the technique of symbolic coding to study the dynamics of pseudo-Anosov maps, which we will adapt for $(\mathcal{L}, \widetilde{f})$. A common strategy to study a given dynamical system is via its semi-conjugacy with a well-understood model. In the context of pseudo-Anosov maps on $S_g$, subshifts of finite type serve as such models. In the case of maps on the ladder surface which are lifts of pseudo-Anosov maps on $S_g$, we exhibit a semi-conjugacy between $(\mathcal{L}, \widetilde{f})$ and a countable Markov shift to understand the dynamical and measure-theoretic properties of $(\mathcal{L}, \widetilde{f})$.

\begin{definition}
 Let $A$ be a matrix with entries from $\{0,1\}$ indexed by a countable set.
     Corresponding to $A$, we have a directed graph $G = (V, E)$.
     Let $\Sigma_A$ be the set of all infinite admissible paths on the graph, which consist of two-sided sequences of vertices.
     Let $\sigma$ be the left shift map on such sequences, that is, $\sigma: \Sigma_A \to \Sigma_A$ is defined for $\mathbf{x} = (x_i)_{i \in \mathbb{Z}}$ as $(\sigma(\mathbf{x}))_i = x_{i+1}$.
\begin{enumerate}[(i)]
    \item The pair $(\Sigma_A, \sigma)$ is known as a \textit{countable Markov shift}.
    \item If $A$ is a finite $n \times n$ matrix, then $(\Sigma_A, \sigma)$ is known as a \textit{subshift of finite type}.
\end{enumerate}
\end{definition}

The concept of Markov partition (see~\cite[Expos\'{e} 10]{flp}) allows for a symbolic coding of a pseudo-Anosov map on a finite type surface, that is, it gives a semi-conjugacy from a subshift of finite type to the dynamical system on a finite type surface with a pseudo-Anosov homeomorphism. However, in the case of an infinite-type surface, we modify the classical definition of Markov partitions by allowing for countably infinitely many partitions of the surface.  

\begin{definition}\label{def:MP}
A \textit{Markov partition} for a map $f$ on a surface $S$ that preserve the transverse pair of measured foliations $(\mathcal{F}_s, \mu_s)$ and $(\mathcal{F}_u, \mu_u)$ with stretch factor $\lambda>1$ is a (finite or countable) collection $\mathcal{R}=\{R_1,R_2,\ldots,R_n,\ldots\}$ of good birectangles constructed using this pair of foliations such that \begin{enumerate}[(i)]
			\item $\operatorname{int}(R_i)\cap \operatorname{int}(R_j)=\emptyset$, for $i\neq j$.
	\item $\bigcup_{i=1}^\infty R_i=S$.
			\item If $x \in \operatorname{int}(R_i)$ and $f(x) \in \operatorname{int}(R_j)$, then      
			\[f(\mathcal{F}^s(x,R_i))\subset \mathcal{F}^s(f(x),R_j), \text{ and } f^{-1}(\mathcal{F}^u(f(x),R_j))\subset \mathcal{F}^u(x,R_i).\]
			\item If $x\in$ $\operatorname{int}(R_i)$ and
			$f(x)\in \operatorname{int}(R_j)$, 
			then \[f(\mathcal{F}^u(x,R_i))\cap R_j=\mathcal{F}^u(f(x),R_j), \text{ and } f^{-1}(\mathcal{F}^s(f(x),R_j))\cap R_i=\mathcal{F}^s(x,R_i).\]
			This means that $f(R_i)$ goes across $R_j$ just once.
		\end{enumerate}
\end{definition}

Note that a good birectangle is an embedding of $[0,1]\times [0,1]$ in the surface whose sides lie alternately in the leaves of two transverse measured foliations (stable and unstable), with opposite sides contained in the parallel leaves of the same foliation, that is, opposite sides are disjoint (see~\cite[Expos\'{e} 10]{flp}). 

By a result of Thurston, there exists a finite Markov partition for any pseudo-Anosov map on the surface $S_g$ which gives a semi-conjugacy from a subshift of finite type to $(S_g,f)$.
 
  \begin{theorem}[Thurston~\cite{flp}]
        For a surface of finite type $S_g$ and a pseudo-Anosov homeomorphism $f: S_g \to S_g$, there exists a finite Markov partition. Consequently, there exists a subshift of finite type $(\Sigma_A, \sigma)$ such that $(\Sigma_A, \sigma)$ is semi-conjugate to a dynamical system $(S_g, f)$, that is, there exists an onto continuous map $\pi: \Sigma_A \to S_g$ such that the following diagram commutes:
        \[\xymatrix{
            & \Sigma_A \ar_\pi[d] \ar^\sigma[r] & \Sigma_A \ar^\pi[d] \\
            & S_g \ar^f[r] & S_g.
        }\]
    \end{theorem}

The topology on $\Sigma_A$ is generated by the cylinder sets defined as follows:
\[[w_{-k} \ldots w_{-1}.w_0 \ldots w_k] = \{ (x_n) \in \Sigma_A : x_{-k}\ldots x_{-1}. x_0 \ldots x_k = w_{-k} \ldots w_{-1}.w_0\ldots w_k \}.\]
With this topology on $\Sigma_A$, the map $\pi$ is continuous. The properties of the matrix $A$ are crucial in understanding the dynamical properties of the dynamical system $(\Sigma_A, \sigma)$, and in turn of $(S_g,f)$ via the semi-conjugacy $\pi$. We now state a few useful definitions from matrix theory.  
    
   \begin{definition}
   Let A be a non-negative integer matrix whose rows and columns are indexed by the same countable set. 
   \begin{enumerate}[(i)]
    \item $A$ is \textit{irreducible} if for every pair of vertices $i$ and $j$ there is an $N\ge 1$ with $(A^N)_{ij}>0$.
    \item For index $i$, the number $p(i)=\text{gcd}\{N : (A^N)_{ii}>0\}$ is the \textit{period} of index $i$. If $A$ is irreducible, then the period of every index is the same and called the \textit{period} of $A$.
    \item A is said to be \textit{aperiodic} if it has period one.
    \end{enumerate}
\end{definition}

The dynamical system $(\Sigma_A, \sigma)$ is topologically transitive and mixing if and only if the matrix $A$ is irreducible and aperiodic. If $A$ is a finite square irreducible and aperiodic matrix, the Perron–Frobenius theorem guarantees the existence of a dominant Perron eigenvalue $\lambda > 1$, which is the spectral radius of $A$. This eigenvalue corresponds to the stretch factor of $f$, along with left and right Perron eigenvectors $\mathbf{u}$ and $\mathbf{v}$, respectively, where each entry of both $\mathbf{u}$ and $\mathbf{v}$ is strictly positive, denoted as $\mathbf{u} > 0$ and $\mathbf{v} > 0$. For the system $(\Sigma_A, \sigma)$, the topological entropy is equal to $\log(\lambda)$. 

For the system $(\Sigma_A, \sigma)$ with $A$ being irreducible and aperiodic, there is a well-known invariant measure known as the \textit{Parry measure}, which is defined using the Perron vectors $\mathbf{u}$ and $\mathbf{v}$ as follows. Let  $\mathbf{u} = (u_1, \ldots, u_m)$ and $\mathbf{v} = (v_1, \ldots, v_m)$, respectively. Define
\begin{eqnarray}\label{eq:P-p}
P_{ij} = \frac{A_{ij} \ v_j}{v_i\ \lambda},
\text{ and }
p_i = \frac{u_i\ v_i}{\sum_{i=1}^m u_i\ v_i}.
\end{eqnarray}
Here, $P$ is a row-stochastic matrix, and $\mathbf{p}$ is its normalized left eigenvector. Now, consider the $\sigma$-algebra $\mathcal{B}$ on $\Sigma_A$ generated by cylinder sets. We define a measure $\mu_A: \mathcal{B} \to \mathbb{R}$ on cylinder sets as
\[
\mu_A([w_0 \ldots w_k]) = p_{w_0} P_{w_0 w_1} \cdots P_{w_{k-1} w_k}.
\]
The measure $\mu_A$ is an invariant probability measure for $\sigma$ and is ergodic. It turns out that $\mu = \mu_A \circ \pi^{-1}$ is an invariant ergodic probability measure on $(S_g, f)$, which is locally equivalent to the product of transverse measures on the stable and unstable foliations. Since $\mu$ is a probability measure, the map $f$ satisfies the Poincar\'e recurrence theorem and Kac's lemma~\cite{YuriPolliDSES}. Our aim is to investigate dynamical properties of $(\mathcal{L}, \widetilde{f})$, for a pseudo-Anosov map $f$, by constructing a symbolic coding via a Markov partition of $(\mathcal{L}, \widetilde{f})$ by lifting a Markov partition of $(S_g,f)$.

\subsubsection{Dynamical Properties}
One of the key features of a pseudo-Anosov map $f$ is that it exhibits rich dynamical properties such as topological transitivity, mixing, and positive topological entropy. Our goal is to examine which of these properties are also exhibited by its lift $\widetilde{f}$. Due to the non-compactness of $\mathcal{L}$, additional care is required when analyzing the dynamics of $\widetilde{f}$, especially for properties like recurrence, which are taken for granted in the setting of compact surfaces due to Poincar\'e recurrence theorem. Thurston~\cite{flp} demonstrates that if $f$ is pseudo-Anosov, then the matrix $A$ is irreducible and aperiodic. Consequently, $(S_g, f)$ is topologically transitive and mixing. This leads to the conclusion that the infinite leaves of $(\mathcal{F}_s, \mu_s)$ and $(\mathcal{F}_u, \mu_u)$ are dense in $S_g$. It also follows that the periodic points of $f$ are dense in $S_g$. Having obtained the invariant foliations on $\mathcal{L}$ (Proposition~\ref{prop:foliation}), our primary objective is to study the dynamical properties of the system $(\mathcal{L}, \widetilde{f})$ analogous to the well-understood system $(S_g, f)$. 

For the system $(\mathcal{L}, \widetilde{f})$, we now turn to the study of its measure-theoretic properties. In the case of $(S_g, f)$, the probability measure $\mu=\mu_A\circ \pi^{-1}$ is ergodic and maximizes the measure-theoretic entropy, that is, it satisfies the variational principle. This measure is built by assigning weights to rectangles defined via the transverse measures on the invariant foliations. In the infinite-type setting, we can similarly define an invariant measure for $\widetilde{f}$ on $\mathcal{L}$, however this measure is $\sigma$-finite rather than finite. We aim to determine for which lifts $\widetilde{f}$ such a measure exists and is ergodic. The non-finiteness of the measure raises several issues, which leads us to investigate whether the system is conservative, and if so, whether it is positive recurrent or null recurrent or transient. A conservative system is defined as follows.

\begin{definition}
    Let $(X,\mathcal{B},\mu)$ be a $\sigma$-finite measure space and 
let $T:X\rightarrow X$ be a measure-preserving transformation. 
Then $T$ is known as \emph{conservative} if there is no wandering set of positive measure. 
that is, there does not exist a measurable set $W\in\mathcal{B}$ with $\mu(W)>0$ such that 
the sets $\{T^{-n}W\}_{n\ge 0}$ are pairwise disjoint 
(see~\cite{inferg}).
\end{definition}

\section{\texorpdfstring{Lifts of maps on ${S}_g$ to the ladder surface}{Lifts of maps on Sg to the ladder surface}}\label{sec:mapping-class}
In this section, we determine a lift of each element in the generating set of $\mathrm{LMod}_{q_g}(S_g)$ (from Theorem~\ref{thm:deygen}) that commutes with $h_{g-1}$. Let $x_0 \in S_g$ be a fixed point of a homeomorphism $f$ representing a generator of $\mathrm{LMod}_{q_g}(S_g)$. Consider a point $\widetilde{x}_0 \in \mathcal{L}$ such that $q_g^{-1}(x_0) = \{h_{g-1}^{i} \widetilde{x}_0\ :\ i \in \mathbb{Z}\}$ and the unique lift $\widetilde{f}$ of $f$ on $\mathcal{L}$ such that $\widetilde{f}(\widetilde{x}_0)= \widetilde{x}_0$, and let $\widetilde{x}_i=h_{g-1}^{i} \widetilde{x}_0$, for all $i\in\mathbb{Z}$. Since the group $\langle h_{g-1} \rangle$ of deck transformations of the cover $q_g: \mathcal{L} \to S_g$ acts transitively on the set of lifts, any other lift of $f$ on $\mathcal{L}$ is of the form $h_{g-1}^i \circ \widetilde{f}$, for some $i \in \mathbb{Z}$. Furthermore, if the lift $\tilde{f}$ commutes with $h_{g-1}$ (i.e., $\widetilde{f}$ belongs to the centralizer $C_{\langle h_{g-1} \rangle}(\mathrm{Homeo}^+(\mathcal{L}))$), then it follows that $\widetilde{f}(\widetilde{x}_i) = \widetilde{x}_i$, for all $i \in \mathbb{Z}$. We prove that if a lift $f'$ of $f$ on $\mathcal{L}$ does not commute with $h_{g-1}$, then ${f'}^2$ (which is a lift of $f^2$) commutes with $h_{g-1}$ (see Corollary~\ref{cor:forf2}). Since we are only concerned about the dynamical properties of the lift, we will assume without loss of generality that the lift $\tilde{f}$ commutes with $h_{g-1}$. 

\subsection{\texorpdfstring{Lifts of the generators of $\mathrm{LMod}_{q_g}(S_g)$}{Lifts of the generators of LMod(S)}}

We recall from Theorem~\ref{thm:deygen} that $\mathrm{LMod}_{q_g}(S_g)$ is generated by the set
	\[\left\lbrace T_{a_1}, T_{a_2}, T_{b_2}, \ldots, T_{a_g}, T_{b_g}, T_{c_1}, \ldots, T_{c_{g-1}}, T_{\alpha_1}T^{-1}_{\iota(\alpha_1)}, \ldots, T_{\alpha_{g-2}}T^{-1}_{\iota(\alpha_{g-2})}, \iota \right\rbrace.
\]
We will now discuss the liftability of each of the generators from this collection. This will be summarized later in Proposition~\ref{prop:gen-lift}.

\subsubsection{Lifting Dehn twists} \label{subsec:dtwist} First, note that a Dehn twist $T_{\gamma}\in \mathrm{Mod}(S_g)$ lifts on $\mathcal{L}$ under $q_g$ if and only if $q_{g}^{-1}(\gamma)=\{\gamma_{i}\ : \ i\in\mathbb{Z}\}$, where each $\gamma_{i}=h_{g-1}^{i}(\gamma_0)$ is an essential simple closed curve. We will prove later that the limit $\lim_{k\to\infty}\prod_{i=-k}^k T_{\gamma_i}$ exists and lies in $\mathrm{Map}({\mathcal{L}})$, and it is a lift of $T_\gamma$. Using this, the lifts of all the Dehn twists $T_{a_j}, T_{b_j}, T_{c_j}$, from the above generating set of $\textstyle{\mathrm{LMod}_{q_g}(S_g)}$ can be visualized. 
\subsubsection{Lifting the hyperelliptic involution} 
The lift of the hyperelliptic involution $\iota$ is given by a $\pi$-rotation of $\mathcal{L}$ about the axis as shown in Figure~\ref{fig:lift-i-L}.
	
	\begin{figure}[htbp]
	\centering
		\begin{tikzpicture}
				\node[anchor=south west,inner sep=0] (image) at (0,0) {\includegraphics[width=0.5\linewidth]{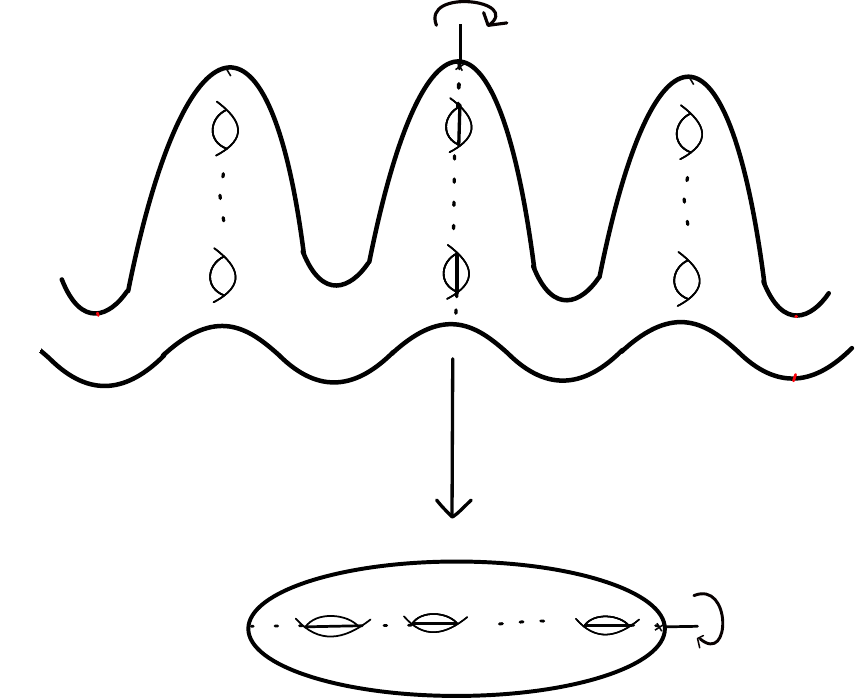}};
				\begin{scope}[x={(image.south east)}, y={(image.north west)}]
				\small
					
					\node[label=above:$\tilde{x}_{-1}$] at (0.27,0.86) {};
					\node[label=above:$\tilde{x}_0$] at (0.48,0.87) {};
					\node[label=above:$\tilde{x}_1$] at (0.8,0.86) {};
					
					\node[label=below:$x_0$] at (0.8,0.11) {};
					
					\node at (0.57,0.35) {$q_g$};
					
					\node at (0.55,1.02) {$\pi$};
					\node at (0.87,0.12) {$\pi$};
					
				\end{scope}
			\end{tikzpicture}
		\caption{The lift $\widetilde{\iota}$ of $\iota$ to $\mathcal{L}$.}
		\label{fig:lift-i-L}
	\end{figure}
\subsubsection{Lifting the bounding pair maps} \label{subsec:bpair} First, we consider a homeomorphism from the mapping class $T_{\alpha_n} T^{-1}_{\iota(\alpha_n)}$. Fix the curve $\alpha_n$ from the isotopy class and consider an annular neighborhood $N$ of $\alpha_n$. Then $\iota(N)$ is an annular neighborhood of $\iota(\alpha_n)$. Consider a parametrization of $N$, $\zeta:[0,2\pi]\times [0,1] \to N$ such that $\zeta([0,2\pi]\times \{1/2\}) = \alpha_n$ and $\zeta((0,1/2))=\zeta((2\pi,1/2))$. Define $\mathcal{T}: [0,2\pi]\times [0,1] \to [0,2\pi]\times [0,1]$ by $\mathcal{T}((\theta,t)) = ((\theta + 2\pi t)\pmod{2\pi},t)$. Similarly, we can define $\iota\circ\zeta:[0,2\pi]\times [0,1] \to \iota(N)$ and $\mathcal{T}_{-1}: [0,2\pi]\times [0,1] \to [0,2\pi]\times [0,1]$ by $\mathcal{T}_{-1}((\theta,t)) = ((\theta - 2\pi t)\pmod{2\pi},t)$. Now we can define $T_{\alpha_n} T^{-1}_{\iota(\alpha_n)}$ as $\zeta \circ \mathcal{T} \circ \zeta^{-1}$ on $N$, $\iota\circ\zeta \circ \mathcal{T}_{-1} \circ (\iota\circ\zeta)^{-1}$ on $\iota(N)$, and the identity map on the remaining surface. This map is homeomorphism as $\zeta \circ \mathcal{T} \circ \zeta^{-1}$ and $\iota\circ\zeta \circ \mathcal{T}_{-1} \circ (\iota\circ\zeta)^{-1}$ are identity on the boundary of $N$ and $\iota(N)$, respectively.

\begin{figure}[htbp]
	\centering
		\begin{tikzpicture}
				
				\node[anchor=south west, inner sep=0] (image) at (0,0) {\includegraphics[width=0.65\linewidth]{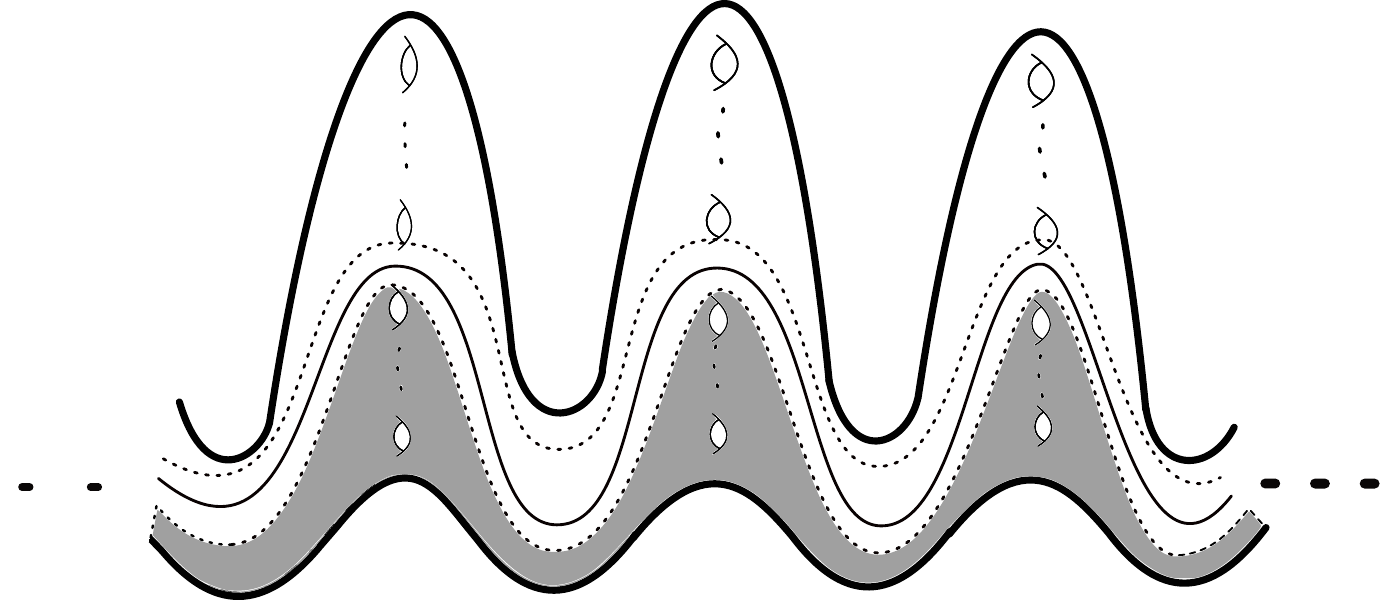}};
				
				\begin{scope}[x={(image.south east)}, y={(image.north west)}]
					
                    \small					
					\node[label=below:$\widetilde{F}_1$] at (0.485,0.35) {};
                    \node[label=below:$\widetilde{F}_2$] at (0.48,0.8) {};
					\node[label=below:$\widetilde{N}$] at (0.4,0.28) {};
                    \node[label=below:$\widetilde{\alpha}_n $] at (0.09,0.33) {};
				\end{scope}
			\end{tikzpicture}
			\caption{The subsurfaces $\widetilde{F}_1$ and $\widetilde{F}_2$ in $\mathcal{L}$.}
			\label{fig:shaded-L}
	\end{figure}

Let $\widetilde{\alpha}_n = q_g^{-1}(\alpha_n)$ and $\widetilde{\iota}(\widetilde{\alpha}_n) = q_g^{-1}(\iota(\alpha_n))$, which are infinite arcs on $\mathcal{L}$. Consider a closed neighborhood $\widetilde{N}=q_g^{-1}(N)$ and $\widetilde{\iota}(\widetilde{N})=q_g^{-1}(\iota(N))$ of $\widetilde{\alpha}_n$ and $\widetilde{\iota}(\widetilde{\alpha}_n)$, respectively. Note that $\mathcal{L}\setminus \widetilde{N}\cup \widetilde{\iota}(\widetilde{N})=\widetilde{F}_1 \sqcup \widetilde{F}_2$, where $\widetilde{F_1}$ is shaded part and $\widetilde{F_2}$ is a non shaded part shown in Figure~\ref{fig:shaded-L}. Consider parametrization $\widetilde{\zeta}:\mathbb{R}\times [0,1] \to \widetilde{N}$ such that for $\widetilde{\theta}:\mathbb{R}\times [0,1] \to [0,2\pi]\times [0,1]$ defined by $\widetilde{\theta}((\theta,t)) = (\theta \pmod{2\pi},t)$, we have $q_g \circ \widetilde{\zeta} = \zeta \circ \widetilde{\theta}$. Now define $\mathcal{T}': \mathbb{R}\times [0,1] \to \mathbb{R}\times [0,1]$ by $\mathcal{T}'((\theta,t)) = ((\theta + 2\pi t),t)$. Similarly, we can define $\iota\circ\widetilde{\zeta}:\mathbb{R}\times [0,1] \to \widetilde{\iota}(\widetilde{N})$ and $\mathcal{T}_{-1}': \mathbb{R}\times [0,1] \to \mathbb{R}\times [0,1]$ by $\mathcal{T}_{-1}'((\theta,t)) = ((\theta - 2\pi t),t)$. Now, we can define the lift of $T_{\alpha_n} T^{-1}_{\iota(\alpha_n)}$ on $\mathcal{L}$ by applying $\widetilde{\zeta} \circ \mathcal{T}' \circ \widetilde{\zeta}^{-1}$ on $\widetilde{N}$, $\iota\circ\widetilde{\zeta} \circ \mathcal{T}_{-1}' \circ (\iota\circ\widetilde{\zeta})^{-1}$ on $\widetilde{\iota}(\widetilde{N})$. Note that $\widetilde{\zeta} \circ \mathcal{T}' \circ \widetilde{\zeta}^{-1}$ and $\iota\circ\widetilde{\zeta} \circ \mathcal{T}_{-1}' \circ (\iota\circ\widetilde{\zeta})^{-1}$ is identity on boundary of $\widetilde{F}_2$ and shift by $2\pi$ on boundary of $\widetilde{F}_1$. Now, apply the $\mathbb{Z}$-action on $\widetilde{F}_1$ and define identity on $\widetilde{F}_2$. It is easy to check that the map defined is a lift of $T_{\alpha_n} T^{-1}_{\iota(\alpha_n)}$.

We summarize the discussion in Sections~\ref{subsec:dtwist}--\ref{subsec:bpair} in the following proposition.

\begin{prop}\label{prop:gen-lift}
 Consider the regular infinite-sheeted cover $q_g:\mathcal{L}\to S_g$ induced by an $\langle h_{g-1}\rangle$-action on $\mathcal{L}$. The following hold:
 \begin{enumerate}[(i)]
 \item A lift of $\iota$ in $\mathrm{Map}(\mathcal{L})$ is represented by $\widetilde{\iota}:\mathcal{L}\to \mathcal{L}$, as described in the preceding discussion. 
 \item A lift of bounding pair $T_{\alpha_n} T^{-1}_{\iota(\alpha_n)} \in \mathrm{LMod}_{q_g}(S_g)$ is represented by $\mathcal{T}_{\widetilde{\alpha}_n} \mathcal{T}'_{\widetilde{\iota}(\widetilde{\alpha}_n)}$, as described in the preceding discussion.
 \item Let $\gamma$ be an essential simple closed curve in $S_g$. If $T_{\gamma}\in \mathrm{LMod}_{q_g}(S_g)$ and $q_g^{-1}(\gamma)=\{\gamma_i\ : \ i\in\mathbb{Z}\}$, then \[\widetilde{T}_\gamma =\lim_{k\to\infty}\prod_{i=-k}^k T_{\gamma_i}.\]  
\end{enumerate} 
\end{prop}

\begin{proof}
Statements $(i)$ and $(ii)$ follow from the discussion in Sections~\ref{subsec:dtwist}--\ref{subsec:bpair}. To prove $(iii)$, we consider a compact exhaustion $(C_n)_{n\ge 1}$ of $\mathcal{L}$, an enumeration $(c_n)_{n\ge 1}$ of the curves in $C(\mathcal{L})$ and a metric $\rho$, following Subsection~\ref{subsec:liftability}. It is straightforward to show that the limit exists in the complete metric space $\left(\text{Map}(\mathcal{L}),\rho\right)$. It remains to be shown that this limit is indeed the lift. The curves $c_1,\dots,c_k$ intersect only finitely many preimages of $\gamma$ under $q_g$. Consider a $K(k)$ such that $c_j\cap \gamma_i=\emptyset$, for $|i|>K(k)+1$. Then for $|I|\ge K(k)$, we have:
\[
\widetilde{T}_\gamma(c_j)=\prod_{i=-I}^{I} T_{\gamma_i}(c_j), \text{ for } j=1,\dots,k\ \Rightarrow \ d\left(\widetilde{T}_\gamma,\prod_{i=-I}^{I} T_{\gamma_i}\right)\leq 2^{-k}.
\] 
By a similar argument, we obtain $K'(k)$ such that for $|I|\ge K'(k)$,
\[
d\left(\widetilde{T}_\gamma^{-1},\left(\prod_{i=-I}^{I} T_{\gamma_i}\right)^{-1}\right)\leq 2^{-k}.
\] 
Hence, for $|I|\ge \max\{K(k),K'(k)\}$, we get:
\[
\rho\left(\widetilde{T}_\gamma,\prod_{i=-I}^{I} T_{\gamma_i}\right)\leq 2^{-k+1},
\]
and we have our assertion.
\end{proof}

A direct consequence of Proposition~\ref{prop:gen-lift} is the following corollary. 

 \begin{cor}\label{cor:seq}
 	Let $\gamma_1,\dots,\gamma_r$ be essential simple closed curves in $S_g$. Suppose that $T_{\gamma_j}\in \mathrm{LMod}_{q_g}(S_g)$ and
    \[
    q_g^{-1}(\gamma_j)=\{\gamma_{j,i}:i\in\mathbb{Z}\}, \qquad j=1,\dots,r.
    \]
    Let $F=T_{\gamma_1}^{n_1}\circ \dots\circ T_{\gamma_r}^{n_r}$, where $n_i \in \mathbb{Z}$, and suppose that $F$ is represented by a homeomorphism $f$ fixing a point $x_0 \in S_g$. Then
 	\[
 	\lim_{k\to\infty}\left(\prod_{i=-k}^k T_{\gamma_{1,i}^{n_1}}\right)\circ \dots\circ\left( \prod_{i=-k}^k T_{\gamma_{r,i}}^{n_r}\right)
 	\]
exists and equals the lift $\widetilde{F} \in \mathrm{Map}(\mathcal{L})$ that is represented by a $\widetilde{f}$ that fixes the point $\widetilde{x}_0 \in q_g^{-1}(x_0) \in \mathcal{L}$. 
\end{cor}

\subsection{Commutativity of lifts with the handle shift map}

We begin this section with the following technical algebraic assertion. 

\begin{lemma}\label{lemma:commlem}
Let $G$ be the generating set for $\mathrm{LMod}_{q_g}(S_g)$ as given in Theorem~\ref{thm:deygen}. For $1\leq j \leq n$, let $F_j \in G$ be represented by $f_j$ such that $F = F_1F_2\ldots F_n$ is represented by $f = f_1f_2\ldots f_n$. Consider a lift $\widetilde{f}=\widetilde{f}_1\widetilde{f}_2\ldots \widetilde{f}_n$ of $f$ on $\mathcal{L}$, where $\widetilde{f}_j$ is a lift of $f_j$ on $\mathcal{L}$. Then one of the following holds. 
\begin{enumerate}[(i)]
\item $h_{g-1}\circ\widetilde{f}=\widetilde{f}\circ h_{g-1}$ if and only if there are even number of $j's$ such that $f_j$ is the hyperelliptic involution $\iota$. 
\item $h_{g-1}\circ\widetilde{f}=\widetilde{f}\circ h_{g-1}^{-1}$ if and only if there are odd number of $j's$ such that $f_j$ is the hyperelliptic involution $\iota$. 
\end{enumerate}
\end{lemma}

\begin{proof}
First, we observe that any lift of a Dehn twist or a bounding pair commutes with $h_{g-1}$, and any lift $\widetilde{\iota}$ of $\iota$ satisfies $h_{g-1}\circ\widetilde{\iota}\circ h_{g-1}=\widetilde{\iota}$. Thus, if $F_j\in G \setminus\{\iota\}$, then $\widetilde{f}_j$  satisfies $h_{g-1}\circ\widetilde{f}_j=\widetilde{f}_j\circ h_{g-1}$. Furthermore, when $f_j=\iota$, we have $h_{g-1}\circ\widetilde{f}_j=\widetilde{f}_j\circ h_{g-1}^{-1}$. Thus, we have \[h_{g-1}\circ \widetilde{f}=h_{g-1}\circ\widetilde{f}_1\widetilde{f}_2\ldots \widetilde{f}_n=\widetilde{f}_1\circ h_{g-1}^{\delta_1}\circ\widetilde{f}_2\ldots \widetilde{f}_n=\cdots=\widetilde{f}_1\widetilde{f}_2\ldots \widetilde{f}_n\circ h_{g-1}^{\delta_1\ldots\delta_n}=\widetilde{f}\circ h_{g-1}^{\delta_1\ldots\delta_n},\]
where the $\delta_j = -1$ (resp. $1$) according as $\widetilde{f}_j = \iota$ (resp. $\neq \iota$). Therefore, it follows that \[h_{g-1}\circ \widetilde{f}_j=\widetilde{f}_j\circ h_{g-1}^{(-1)^k},\] where $k$ is the number of times $\iota$ appears in the product $\widetilde{f}_1\widetilde{f}_2\ldots \widetilde{f}_n$. 

Since any representation of $f$ as a product of elements in $G$, the number of times $\iota$ appears is either always even or always odd, our assertion will now follow from the preceding argument.
\end{proof}

An immediate consequence of Lemma~\ref{lemma:commlem} is the following proposition.

  \begin{prop}
         Let $F\in\mathrm{LMod}_{q_g}(S_g)$ be represented by an $f$ that lifts to a homeomorphism $\widetilde{f}$ on $\mathcal{L}$. Then one of the following possibilities occurs:
         \begin{enumerate}[(i)]
             \item $h_{g-1}\circ\widetilde{f}\circ h_{g-1}^{-1}=\widetilde{f}.$
             \item $h_{g-1}\circ\widetilde{f}\circ h_{g-1}=\widetilde{f}.$
            \end{enumerate}
    \end{prop}
    
\begin{cor}
\label{cor:forf2}
Let $F\in\mathrm{LMod}_{q_g}(S_g)$ be represented by an $f$ that lifts to a homeomorphism $\widetilde{f}$ on $\mathcal{L}$. If $\widetilde{f}$ satisfies $h_{g-1}\circ\widetilde{f}\circ h_{g-1}=\widetilde{f}$, then the lift $\widetilde{f}^2$ of $f^2$ satisfies $h_{g-1}\circ\widetilde{f}^2\circ h_{g-1}^{-1}=\widetilde{f}^2$. 
\end{cor}

Since the dynamical properties of our interest are shared by $\widetilde{f}$ and $\widetilde{f}^2$, it would not alter our analysis if we choose to consider one in place of the other. Hence, we will assume that the lift $\widetilde{f}$ commutes with $h_{g-1}$.
 
\section{\texorpdfstring{Dynamics of lifts of pseudo-Anosov maps on $S_g$ to the ladder surface}{Dynamics of lifts of pseudo-Anosov maps on Sg to the ladder surface}}

\subsection{\texorpdfstring{Lifting the symbolic coding from $S_g$ to $\mathcal{L}$}{Lifting the symbolic coding from Sg to the ladder surface}}\label{subsec:symcodingsec}

We will consider a countably infinite collection of rectangles in our definition of a Markov partition (see Definition~\ref{def:MP}). By lifting the Markov partition from $(S_g, f)$ to $(\mathcal{L}, \widetilde{f})$, we obtain a countable Markov partition of $\mathcal{L}$, which results in a countable Markov shift $(\widetilde{\Sigma}_{\widetilde{A}}, \sigma)$, where $\widetilde{A}$ is a countable $0-1$ matrix. This allows us to establish a semi-conjugacy from the Markov shift $(\widetilde{\Sigma}_{\widetilde{A}}, \sigma)$ to the dynamical system $(\mathcal{L}, \widetilde{f})$. Specifically, we produce a map $\widetilde{\pi}: \widetilde{\Sigma}_{\widetilde{A}} \to \mathcal{L}$ which is a surjective continuous and finite-to-one map. We present the following theorem which we will prove using a series of four lemmas.

\begin{theorem}\label{thm:conj}
Consider a pseudo-Anosov homeomorphism $f:S_g\rightarrow S_g$ ($g\ge 3$) and its lift $\widetilde{f}:\mathcal{L}\rightarrow\mathcal{L}$ under the regular-sheeted cover $q_g:\mathcal{L}\rightarrow S_g$. Then there exists a countable Markov partition $\widetilde{\mathcal{R}}$ for $(\mathcal{L},\widetilde{f})$. Consequently, there exists a countable Markov shift $(\widetilde{\Sigma}_{\widetilde{A}}, \sigma)$ such that the dynamical system $(\widetilde{\Sigma}_{\widetilde{A}}, \sigma)$ is semi-conjugate to $(\mathcal{L}, \widetilde{f})$, that is, there exists a surjective continuous map $\widetilde{\pi}: \widetilde{\Sigma}_{\widetilde{A}} \to\mathcal{L}$ such that the following diagram commutes:
    \[
     \xymatrix{
            & \widetilde{\Sigma}_{\widetilde{A}} \ar_{\widetilde{\pi}}[d] \ar^\sigma[r] & \widetilde{\Sigma}_{\widetilde{A}} \ar^{\widetilde{\pi}}[d] \\
            & \mathcal{L} \ar^{\widetilde{f}}[r] & \mathcal{L}.}
    \]
Furthermore, $\widetilde{\pi}$ is a finite-to-one map.
\end{theorem}

The result follows through a series of lemmas. In the following lemmas, assume the hypotheses of the preceding theorem.
     
    \begin{lemma}
    	There exists a countable Markov partition for  the dynamical system $(\mathcal{L},\widetilde{f})$.
    \end{lemma}
    
\begin{proof}
Let $\mathcal{R}$ be a finite Markov partition of $(S_g,f)$ consisting of closed rectangles. Let $x_0$ be a fixed point of $f$ (recall the discussion about the fixed point from Section~\ref{sec:mapping-class}). Let us now label the elements of $\mathcal{R}$ as per the following procedure. Let $R_1$ be a rectangle which contains the point $x_0$. Then, choose a rectangle that intersects $R_1$ and label it as $R_2$. By induction on $n$ for $2\leq n\leq m$, choose a rectangle that intersects $\cup_{i=1}^{n-1}R_i$ and label it as $R_n$. This is possible since $S_g$ is connected and also $\cup_{i=1}^{n-1}R_i$ is connected.  

Let $\widetilde{x}_0\in\mathcal{L}$ be a preimage of $x_0$ under $q_g$. Let $\widetilde{R}_{(1,0)}$ be a connected preimage of $R_1$ under $q_g$ containing $\widetilde{x}_0$ such that $\widetilde{R}_{(1,0)}$ is bijective with $R_1$ under $q_g$. Let $\widetilde{R}_{(2,0)}$ be the connected preimage of $R_2$ under $q_g$ such that $\widetilde{R}_{(2,0)}$ is bijective with $R_2$ under $q_g$ and intersects $\widetilde{R}_{(1,0)}$. This is possible since $R_2$ intersects $R_1$. By induction on $n$, for $2\leq n\leq m$, let $\widetilde{R}_{(n,0)}$ be the connected preimage of $R_n$ under $q_g$ such that $\widetilde{R}_{(n,0)}$ is bijective with $R_n$ under $q_g$ and intersects $\cup_{i=1}^{n-1}\widetilde{R}_{(i,0)}$. This is possible since $R_n$ intersects $\cup_{i=1}^{n-1}R_i$. Now we will extend the preimages $\{R_{n,0}\}_{1\le n\le m}$ to a Markov partition of $(\mathcal{L},\widetilde{f})$. 

Set $K_0=\cup_{i=1}^{m}\widetilde{R}_{(i,0)}$ and let $K_j=h_{g-1}^j(K_0)$, for all $j\in\mathbb{Z}$. Observe that $K_j$ is a connected subsurface of $\mathcal{L}$ such that $q_g(K_j)=S_g$, $q_g$ is injective in the interior of $K_j$, and $\widetilde{x}_j=h_{g-1}^j(\widetilde{x}_0)\in K_j$, for all $j\in\mathbb{Z}$. Furthermore, $\mathcal{L}=\cup_{j\in \mathbb{Z}}K_j$. Set $R_{(i,j)}=h_{g-1}^j(R_{(i,0)})$, for all $1\leq i\leq m$ and $j\in\mathbb{Z}$. For each $1\le i\le m$, $q_g^{-1}(R_i)$ is a countable disjoint union $\{h^{j}_{g-1}(R_{(i,j)})\}_{j\in\mathbb{Z}}$. Note that each $R_{(i,j)}$ is a connected subset of $q_g^{-1}(R_i)$ and is bijective with $R_i$ due to the property of the covering map, that is, $q_g^{-1}(y)=\{h^{j}_{g-1}(\widetilde{y}_0)\}_{j\in \mathbb{Z}}$ for $\forall y\in S_g$ and $\widetilde{y}_0\in q_g^{-1}(x)$. 

We will now show that the collection 
\[
\widetilde{\mathcal{R}}=\{R_{(i,j)}\ : \ i=1,\ldots,m, \ j\in \mathbb{Z}\},
\]
is a Markov partition of $(\mathcal{L}, \widetilde{f})$.

We first establish that each $R_{(i,j)}$ is a good birectangle. Recall that we defined the lifted stable and unstable foliations on $\mathcal{L}$ as the pullback of the stable and unstable foliations on $S_g$. Further since $R_i$ is a good birectangle, the boundary of $R_i$ consists of two disjoint stable segments and two disjoint unstable segments. Thus, the preimage of each stable segment under $q_g$ is a union of stable segments in $\mathcal{L}$, and the preimage of each unstable segment under $q_g$ is a union of unstable segments in $\mathcal{L}$. Moreover since $R_{(i,j)}$ is a connected component of $q_g^{-1}(R_i)$, it must be bounded by exactly two disjoint stable segments and two disjoint unstable segments, which are the preimages of the corresponding segments in $R_i$. Thus, each $R(i,j)$ is a closed rectangle and opposite sides are disjoint. Therefore, each $R_{(i,j)}$ is a good birectangle. 

Now we verify all the conditions of Markov partition for $\widetilde{\mathcal{R}}=\{R_{(i,j)}\}$. At each step, we use that fact that $\mathcal{R}$ is a Markov partition.

\begin{enumerate}[(i)]
    \item For each $1\leq i_1,i_2\leq m$, since $\operatorname{int} R_{i_1} \cap \operatorname{int} R_{i_2} = \emptyset$, we have $\operatorname{int} R_{(i_1,j_1)} \cap \operatorname{int} R_{(i_2,j_2)} = \emptyset$, for all $j_1,j_2\in\mathbb{Z}$.
    \item Applying $q_g^{-1}$ to $\cup_{i=1}^m R_i = S_g$, we get $\cup_{i=1}^m \cup_{j\in\mathbb{Z}} R(i,j) = \mathcal{L}$.
    \item If $\widetilde{x}\in \operatorname{int} R_{(i,j)}$ and $\widetilde{f}(\widetilde{x})\in \operatorname{int} R_{(i',j')}$, then $x=q_g(\widetilde{x})\in \operatorname{int} R_i$ and $f(x)\in \operatorname{int} R_{i'}$. \\
    Further since $f(\mathcal{F}^s(x, R_i)) \subset \mathcal{F}^s(f(x), R_{i'})$, we get $\widetilde{f}(\widetilde{\mathcal{F}}^s(\widetilde{x}, R_{(i,j)})) \subset q^{-1}_g(\mathcal{F}^s(f(x), R_{i'}))$. \\
    Since $\{R_{(i',j'')}\}_{j''\in\mathbb{Z}}$ is a disjoint collection of good birectangles, we get \[
    q^{-1}_g(\mathcal{F}^s(f(x), R_{i'}))=\bigsqcup_{j''\in\mathbb{Z}}\widetilde{\mathcal{F}}^s(h^{j''}(\widetilde{f}(\widetilde{x})),  R_{(i',j'')}).
    \]
    Since $\widetilde{f}(\widetilde{x})\in \operatorname{int} R_{(i',j')}$, we get
    \[
    \widetilde{f}(\widetilde{x})\in \widetilde{f}(\widetilde{\mathcal{F}}^s(\widetilde{x}, R_{(i,j)}))\cap \widetilde{\mathcal{F}}^s(\widetilde{f}(\widetilde{x}),  R_{(i',j')})\neq \emptyset. 
    \]
    Furthermore, since $\widetilde{f}$ is a homeomorphism, $\widetilde{f}(\widetilde{\mathcal{F}}^s(\widetilde{x}, R_{(i,j)}))$ is connected stable segment and
    \[
    \widetilde{f}(\widetilde{\mathcal{F}}^s(\widetilde{x}, R_{(i,j)})) \subset \widetilde{\mathcal{F}}^s(\widetilde{f}(\widetilde{x}),  R_{(i',j')}).
    \] 
    \noindent
    Similarly, we can show that $\widetilde{f}^{-1}(\widetilde{\mathcal{F}}^u(\widetilde{f}(\widetilde{x}),  R_{(i',j')}))\subset\widetilde{\mathcal{F}}^u(\widetilde{x}, R_{(i,j)}) $.
    \item The arguments for this part are similar to those of (iii). Since $f(\mathcal{F}^u(x, R_i)) \cap R_{i'} = \mathcal{F}^u(f(x), R_{i'})$, we have 
    \[
    \widetilde{f}(\widetilde{\mathcal{F}}^u(\widetilde{x}, R_{(i,j)}))\cap R_{(i',j')} \subset q^{-1}_g(f(\mathcal{F}^u(x, R_i)) \cap R_{i'})=q^{-1}_g(\mathcal{F}^u(f(x), R_{i'})).
    \] 
    Moreover
    \[
    q^{-1}_g(\mathcal{F}^u(f(x), R_{i'}))=\bigsqcup_{j''\in\mathbb{Z}}\widetilde{\mathcal{F}}^u(h^{j''}(\widetilde{f}(\widetilde{x})), R_{(i',j'')}). 
    \]
    Since $\widetilde{f}(\widetilde{x})\in \operatorname{int} R_{(i',j')}$, we get \[
    \widetilde{f}(\widetilde{x})\in \widetilde{f}(\widetilde{\mathcal{F}}^s(\widetilde{x}, R_{(i,j)}))\cap \widetilde{\mathcal{F}}^s(\widetilde{f}(\widetilde{x}),  R_{(i',j')})\neq \emptyset. 
    \]
    Furthermore, $\widetilde{f}(\widetilde{\mathcal{F}}^u(\widetilde{x}, R_{(i,j)}))$ is connected unstable segment and \[\widetilde{f}(\widetilde{\mathcal{F}}^u(\widetilde{x}, R_{(i,j)})) \cap R_{(i',j')} \subset \widetilde{\mathcal{F}}^u(\widetilde{f}(\widetilde{x}),  R_{(i',j')}).\] 
    Since $q_g$ is a homeomorphism on $R_{(i',j')}$, we get $\widetilde{f}(\widetilde{\mathcal{F}}^u(\widetilde{x}, R_{(i,j)}) \cap R_{(i',j')}) = \widetilde{\mathcal{F}}^u(\widetilde{f}(\widetilde{x}),  R_{(i',j')})$.
    \noindent
    Similarly, we can show that $\widetilde{f}^{-1}(\widetilde{\mathcal{F}}^s(\widetilde{f}(\widetilde{x}),  R_{(i',j')}))\cap R_{(i,j)} = \widetilde{\mathcal{F}}^s(\widetilde{x}, R_{(i,j)}) $.
\end{enumerate}

\end{proof}

\begin{lemma}\label{lemma:uniqueint}
If $f(\operatorname{int}(R_{\alpha})) \cap \operatorname{int}(R_{\beta}) \neq \emptyset$, then 
\begin{enumerate}[(i)]
	\item there exists a unique $j'=j'(\alpha,\beta)\in\mathbb{Z}$ such that $\widetilde{f}(\operatorname{int}(R_{(\alpha,0)})) \cap \operatorname{int}(R_{(\beta,j')}) \neq \emptyset$. Consequently, for each $j\in\mathbb{Z}$, $j+j'$ is the unique integer for which $\widetilde{f}(\operatorname{int}(R_{(\alpha,j)})) \cap \operatorname{int}(R_{(\beta,j+j')}) \neq \emptyset$.
	\item there exists a unique $j'=j'(\alpha,\beta)\in\mathbb{Z}$ such that $\widetilde{f}(\operatorname{int}(R_{(\alpha,j')})) \cap \operatorname{int}(R_{(\beta,0)}) \neq \emptyset$. Consequently, for each $j\in\mathbb{Z}$, $j+j'$ is the unique integer for which $\widetilde{f}(\operatorname{int}(R_{(\alpha,j+j')})) \cap \operatorname{int}(R_{(\beta,j)}) \neq \emptyset$.
\end{enumerate} 
\end{lemma}

\begin{proof}
We prove the first statement, the second one follows similarly. Since $f(\operatorname{int}(R_{\alpha})) \cap \operatorname{int}(R_{\beta}) \neq \emptyset$ and $\widetilde{f}$ is a lift of $f$, there exists at least one $j' \in \mathbb{Z}$ such that $\widetilde{f}(\operatorname{int}(R_{(\alpha,0)})) \cap \operatorname{int}(R_{(\beta,j')}) \neq \emptyset$. We now prove uniqueness of $j'$. On the contrary suppose there are two such $j'$ and $j''$. Set $C'=\widetilde{f}(\operatorname{int}(R_{(\alpha,0)})) \cap \operatorname{int}(R_{(\beta,j')})$ and $C''=\widetilde{f}(\operatorname{int}(R_{(\alpha,0)})) \cap \operatorname{int}(R_{(\beta,j'')})$. Clearly $C',C''$ are connected, $C'\cap C''=\emptyset$ and $C'\cup C''$ is disconnected. Now, $\widetilde{f}^{-1}(C'\cup C'')$ is a subset of $\operatorname{int}(R_{(\alpha,0)})$ and is disconnected. Since $q_g$ is a homeomorphism on $\operatorname{int}(R_{(\alpha,0)})$, $q_g(\widetilde{f}^{-1}(C'\cup C''))$ is disconnected in $S_g$. Since
\[
q_g \widetilde{f}^{-1}(C' \cup C'') = f^{-1} q_g(C' \cup C'') = \operatorname{int} (R_\alpha) \cap f^{-1}(\operatorname{int} (R_\beta)).
\]
and $\operatorname{int} (R_\alpha) \cap f^{-1}(\operatorname{int} (R_\beta))$ is connected, we have arrived at a contradiction. Thus, there is a unique $j'\in \mathbb{Z}$ as desired. The final statement of (i) follows since $h_g$ and $\widetilde{f}$ are commuting homeomorphisms.  
\end{proof}

\noindent \textbf{Definition of the countable Markov shift $(\widetilde{\Sigma}_{\widetilde{A}},\sigma)$}:\\
Define a countably infinite $0-1$ matrix $\widetilde{A}$ indexed by the tuples $\widetilde{\Sigma}=\{(i,j)\ : \ i=1,\dots,m,\ j\in\mathbb{Z}\}$ as follows: 
\[
\widetilde{A}_{(i_1,j_1)(i_2,j_2)} =
\begin{cases}
	1, & \text{if } \widetilde{f}(\operatorname{int}(R_{(i_1,j_1)}) \cap \operatorname{int}(R_{(i_2,j_2)}) \neq \emptyset, \\
	0, & \text{otherwise}.
\end{cases}
\]
This defines a countable Markov shift $(\widetilde{\Sigma}_{\widetilde{A}},\sigma)$. Any sequence $(i_n,j_n)_n\in \widetilde{\Sigma}_{\widetilde{A}}$ can be expressed as a function $\phi=(\phi_1,\phi_2):\mathbb{Z}\to \{1,\dots,m\}\times \mathbb{Z}$ such that $\phi(n)=(\phi_1(n),\phi_2(n))=(i_n,j_n)$, for all $n\in\mathbb{Z}$. This notation will be used in the following lemma.

\begin{rem}
	A quick observation from the above results is that $\widetilde{A}_{(i_1,j_1)(i_2,j_2)}=1$ implies $A_{i_1,i_2}=1$. Conversely, $A_{i_1,i_2}=1$ implies for each $j_1$, there exists a unique $j_2$ such that $\widetilde{A}_{(i_1,j_1)(i_2,j_2)}=1$. 
\end{rem}

\begin{lemma}\label{lemma:uniqseq}
Given $(i_n)_n\in\Sigma_{{A}}$ and $j_0\in\mathbb{Z}$, there exists a unique sequence $\phi=(\phi_1,\phi_2)$ in $\widetilde{\Sigma}_{\widetilde{A}}$ for which $\phi_1(n)=i_n$, for $n\in\mathbb{Z}$ and $\phi_2(0)=j_0$. We will denote this sequence by $\phi_{{(i_n)_n},j_0}$.
\end{lemma}

\begin{proof}
	We are given that $\phi(0)=(\phi_1(0),\phi_2(0))=(i_0,j_0)$. Since $f(\operatorname{int}(R_{i_0})) \cap \operatorname{int}(R_{i_1}) \neq \emptyset$, by Lemma~\ref{lemma:uniqueint}, there exists a unique $j_1=j_0+j(i_0,i_1)$ such that $\widetilde{f}(\operatorname{int}(R_{(i_0,j_0)}) \cap \operatorname{int}(R_{(i_1,j_1)}) \neq \emptyset$. Thus, $\phi(1)=(\phi_1(1),\phi_2(1))=(i_1,j_1)$. By induction on $n\ge 2$, there is a unique sequence ${(j_n)_n}$ of integers such that $\widetilde{f}(\operatorname{int}(R_{(i_n,j_n)}) \cap \operatorname{int}(R_{(i_{n+1},j_{n+1})}) \neq \emptyset$.
	
	Furthermore, since $f(\operatorname{int}(R_{i_{-1}})) \cap \operatorname{int}(R_{i_0}) \neq \emptyset$, by Lemma~\ref{lemma:uniqueint}, there exists a unique $j_{-1}=j_0+j(i_{-1},i_0)$ such that $\widetilde{f}(\operatorname{int}(R_{(i_{-1},j_{-1})}) \cap \operatorname{int}(R_{(i_0,j_0)}) \neq \emptyset$. Thus, $\phi(-1)=(i_{-1},j_{-1})$. By induction on $n\ge 2$, there is a unique sequence ${(j_{-n})_n}$ of integers such that $\widetilde{f}(\operatorname{int}(R_{(i_{-n},j_{-n})}) \cap \operatorname{int}(R_{(i_{-n+1},j_{-n+1})}) \neq \emptyset$. Hence, there is a unique sequence $\phi$ as desired.  
\end{proof}

\begin{lemma}
	There exists a semi-conjugacy $\widetilde{\pi}: (\widetilde{\Sigma}_{\widetilde{A}},\sigma) \to 
	(\mathcal{L},\widetilde{f})$.
\end{lemma}

\begin{proof}
	Define $\widetilde{\pi}:\widetilde{\Sigma}_{\widetilde{A}} \to \mathcal{L}$ as follows: 
	\[
	\widetilde{\pi}((i_n,j_n)_n)=\bigcap_{n\in\mathbb{Z}}\widetilde{f}^{-n}(R_{(i_n,j_n)}).
	\]
	The map $\widetilde{\pi}$ is well-defined since $\cap_{n\in\mathbb{Z}}\widetilde{f}^{-n}(R_{(i_n,j_n)})$ is a singleton for each $(i_n,j_n)_n\in \widetilde{\Sigma}_{\widetilde{A}}$ since $\widetilde{\mathcal{R}}$ is a Markov partition. Moreover, $\widetilde{\pi}$ is continuous with the product topology on $\widetilde{\Sigma}_{\widetilde{A}}$. Further $\widetilde{\pi}$ is surjective since $\pi:\Sigma_{A}\to S_g$ is surjective. Finally, 
	\[\widetilde{\pi}(\sigma((i_n,j_n)_n))= \widetilde{\pi}((i_{n+1},j_{n+1})_n)=\bigcap_{n\in\mathbb{Z}}\widetilde{f}^{-n}(R_{(i_{n+1},j_{n+1})})
	\]
	\[=\widetilde{f}\left(\bigcap_{n\in\mathbb{Z}}\widetilde{f}^{-(n+1)}(R_{(i_{n+1},j_{n+1})})\right)=\widetilde{f}(\widetilde{\pi}((i_n,j_n)_n)),
	\]
	That is,  $\widetilde{f}\circ\widetilde{\pi}=\widetilde{\pi}\circ \sigma.$  Thus $\widetilde{\pi}$ is a semi-conjugacy.
\end{proof}

\begin{lemma}
The map $\widetilde{\pi}$ is finite-to-one.
\end{lemma}
\begin{proof}
Let $x \in K_j$. Since the map $\pi$ is finite-to-one, we have  
\[
\pi^{-1}(q_g(x)) = \left\{ (i_n^1)_n, (i_n^2)_n, \dots, (i_n^k)_n \right\} \subset \Sigma_A.
\]  
We claim that  
\[
\widetilde{\pi}^{-1}(x) = \left\{ \phi_{(i_n^1)_n,j}, \phi_{(i_n^2)_n,j}, \dots, \phi_{(i_n^k)_n,j} \right\} \subset \widetilde{\Sigma}_{\widetilde{A}}.
\]
Since $x \in K_j$ and $q_g(x) \in R_{i_0^\ell}$, for only $\ell = 1, \dots, k$, we have that $x \in R_{(i_0^\ell,j)}$. Thus, any sequence $\phi$ in $\widetilde{\pi}^{-1}(x)$ satisfies $\phi_2(0) = j$ and $\phi_1(0) = i_0^\ell$, for $\ell=1,\dots,k$. Consider one particular sequence $\phi_{(i_n^\ell)_n,j}$. Since $q_g(\widetilde{f}^n(x)) = f^n(q_g(x))$, we have $\phi_1(n) = i_n^\ell$, for all $n\in\mathbb{Z}$. By Lemma~\ref{lemma:uniqseq}, there exists a unique $\phi_{(i_n^\ell)_n,j}$ such that $\widetilde{\pi}(\phi_{(i_n^\ell)_n,j}) = x$. This is the case for each $\ell=1,\dots,k$. Hence, we have proved our claim.
\end{proof}

The indexing of rectangles of the Markov partition $\widetilde{\mathcal{R}}$ allows us to get a nice structure of matrix $\widetilde{A}$ in Section~\ref{sec:QBD} and to get the desired properties of the map $\widetilde{\pi}$.

The collection of all the possible lifts of $f$ on $\mathcal{L}$ under $q_g$ is given by $\{h_g^n\widetilde{f}\ : \ n\in\mathbb{Z}\}$. In the following section, we use the symbolic coding of $(\mathcal{L},\widetilde{f})$ to study the dynamical properties of all the lifts. 

\subsection{\texorpdfstring{Structure of the transition matrix $\widetilde{A}$}{Structure of the transition matrix A-tilde}}\label{sec:QBD}
In this section, we will first obtain the structure of the transition matrix $\widetilde{A}$ and then we will show that given irreducibility, $\widetilde{A}$ can never be positive recurrent. Since $K_0$ is compact, $\widetilde{f}(K_0)$ is also compact. Therefore, there exists $k\ge 1$ such that $\widetilde{f}(K_0) \subset \cup_{j=-k}^{k} K_{j}$. Set $g_k=(g-1)(2k+1)+1$, for $k\ge 1$. 

Since $\mathrm{LMod}_{q_g}(S_g) = \cap_{k \geq 2} \mathrm{LMod}_{p_k}(S_g)$, then there exists a pseudo-Anosov lift $\widetilde{f}_{k}: (S_{g_{k}},\widetilde{x}_0^k) \to (S_{g_{k}},\widetilde{x}_0^k)$ of the pseudo-Anosov map $f:(S_g,x_0)\to(S_g,x_0)$ with $p_k(\widetilde{x}_0^k)=x_0$ and $q_{g_k}(\widetilde{x}_0)=\widetilde{x}_0^k$. Consequently, $\widetilde{f}$ is a lift of $\widetilde{f}_{k}$ and $q_{g_{k}}(\cup_{i=-k}^{k} K_{i}) = S_{g_{k}}$. Observe that $\{R_{(i,j)}\ : \ i=1,\dots,m,\ j=-k,\dots,k\}$ is a Markov partition of $f_{k}$ with an irreducible and aperiodic transition matrix $A_k$, say. Since $\widetilde{f}(K_0) \subset \cup_{j=-k}^{k} K_{j}$ and $h_g$, $\widetilde{f}$ commute, we have 
\[
\widetilde{f}(\cup_{j=-k}^{k} K_{j}) \subset \cup_{j=-2k}^{2k} K_{j}.
\]
We now redefine the notations to obtain a nice structure for the matrix $\widetilde{A}$. Use the map $f_k$ on $S_{g_{k}}$ in place of $f$ on $S_g$, $\cup_{i=-k}^{k} K_{i}$ in place of $K_0$, $m(2k+1)$ in place of $m$, $g_k$ in place of $g$, and $A_k$ in place of $A$. Think of $\widetilde{f}$ as a lift of $f_{k}$ (which we will now denote as $f$) from now on. With the new notations, let $K_n=h_{g-1}^n(K_0)$, for $n\in\mathbb{Z}$. Observe that 
\[
\widetilde{f}(K_0) \subset K_{-1}\cup K_0\cup K_1. 
\]

\begin{rem}\label{rem:club}
We will use a similar argument of using the map $f_k$ on $S_{g_{k}}$ in place of $f$ on $S_g$, and considering the lift $\widetilde{f}$ of $f_k$ to obtain nearest-neighbor connections only in Section~\ref{subsec:TopTran} for neighborhoods as well. It will give a tridiagonal block structure for the train track matrix like it does for the matrix $\widetilde{A}$ as we shall observe in the upcoming discussion. 
\end{rem}

Using the results from the preceding section, we obtain three matrices $A_{-1}, A_0, A_1$ of order $m$ given as follows: for $\alpha,\beta\in\{1,\dots,m\}$ and $j=-1,0,1$,
\[
(A_{j})_{\alpha,\beta}=\begin{cases}
	1, & \widetilde{f}(\operatorname{int}( R_{(\alpha,0)})) \cap \operatorname{int}( R_{(\beta,j)}) \neq \emptyset,\\
	0, & \text{ otherwise}.
\end{cases}
\]
Moreover, by Lemma~\ref{lemma:uniqueint}, we have
\[
A_{\alpha,\beta}=(A_{-1})_{\alpha,\beta}+(A_{0})_{\alpha,\beta}+(A_{1})_{\alpha,\beta}.
\]
Further, because $h_{g-1}$ commutes with $\widetilde{f}$, for $j=-1,0,1$ we have
\[
\widetilde{f}(\operatorname{int} R_{(\alpha,0)})\cap \operatorname{int} R_{(\beta,j)}\neq\emptyset
\quad\Longleftrightarrow\quad
\widetilde{f}(\operatorname{int} R_{(\alpha,\ell)})\cap \operatorname{int} R_{(\beta,\ell+j)}\neq\emptyset.
\]
Hence, the transition matrix $\widetilde{A}$ for the lift $\widetilde{f}$ and the Markov partition $\widetilde{R}$ has the form below. Its rows and columns are indexed by
$\widetilde{\Sigma}=\{(i,j):i=1,\ldots,m,\ j\in\mathbb{Z}\}$. We order this set lexicographically by the second coordinate: $(i,j)<(i',j')$ if $j<j'$, or if $j=j'$ and $i<i'$.

\begin{lemma}
The transition matrix $\widetilde{A}$ for the Markov partition $\widetilde{\mathcal{R}}$ has the following tridiagonal block form:
\[
\widetilde{A} = {
\begin{bmatrix}
\ddots&\vdots & \vdots & \vdots & \vdots & \vdots  \\
\dots& A_0 & A_1 & 0 & 0 & \dots \\
\dots& A_{-1} & A_0 & A_1 & 0 & \dots \\
\dots &0& A_{-1} & A_0 & A_1 & \dots \\
\dots &0&0& A_{-1} & A_0  & \dots \\
\vdots & \vdots & \vdots & \vdots & \vdots & \ddots
\end{bmatrix}},
\]
where $A_{-1}$, $A_{0}$ and $A_1$ are non-zero matrices of size $m$ with $A_{-1}+A_0+A_1=A$. 
\end{lemma}

\begin{rem}\label{rem:block-diag}
(Other lifts of $f$ on $\mathcal{L}$ under $q_g$) Any other lift of $f$ on $\mathcal{L}$ under the cover $q_g$ is given by $h_{g-1}^j \widetilde{f}$ for nonzero $j \in \mathbb{Z}$. Note that $\widetilde{R}$ is a Markov partition for every lift. The transition matrix for the Markov partition $\widetilde{R}$ has the same structure as that of $\widetilde{A}$ with the tridiagonal shifted by $j$ units (to the right if $j$ is positive and to the left if $j$ is negative). Hence, it is not irreducible for $j\ne 0$. The corresponding map $h_{g-1}^j \widetilde{f}$ just translates the surface to one of the ends depending on the sign of $j$ and does not possess properties such as topological mixing, which we seek. As a result, these mappings are not good representatives of pseudo-Anosov-like maps. Consequently, from now on, we will focus solely on $\widetilde{f}$ because it is the only map for which there is a possibility that the matrix $\widetilde{A}$ is irreducible and aperiodic. 
 \end{rem}
 
 \subsection{\texorpdfstring{Deriving a stochastic process from the countable Markov shift $(\widetilde{\Sigma}_{\widetilde{A}},\sigma)$}{Deriving a stochastic process from the countable Markov shift}}

In this section, we will first obtain a stochastic matrix $\widetilde{P}$ from the transition matrix $\widetilde{A}$ using the Perron-Frobenius theorem in a similar fashion as for the stochastic matrix $P$ compatible with $A$ obtained in~\eqref{eq:P-p}. The Perron-Frobenius theorem holds for a matrix with a countably infinite index set as well, however, with an additional assumption of the matrix being recurrent and positive recurrent (see~\cite[Theorem 7.1.3.]{kit}). We use this result for the infinite matrix $\widetilde{A}$. We will first show that the largest eigenvalue $\lambda'$ in modulus of $\widetilde{A}$ is equal to the maximal eigenvalue $\lambda$ of $A$ (which is the same as the stretch factor of $f$). Observe that  $(\widetilde{A}^n)_{(i,j),(i',j')}\leq (A^n)_{i,i'}$, for all $i=1,\dots,m$, and $j\in\mathbb{Z}$. Thus, \[
\lambda'=\lim_{n\to\infty}((\widetilde{A}^n)_{(i,j),(i,j)})^{\frac{1}{n}}\leq \lim_{n\to\infty}((A^n)_{i,i})^{\frac{1}{n}}=\lambda.
\]
Further since $A_{-1}+A_0+A_1=A$, we get that $\lambda$ is an eigenvalue of $\widetilde{A}$ with  $\widetilde{\mathbf{u}}=[\dots,\mathbf{u},\mathbf{u},\mathbf{u},\dots]$ as the (positive) left eigenvector of $\widetilde{A}$ and $\widetilde{\mathbf{v}}=[\dots,\mathbf{v}^{\top},\mathbf{v}^{\top},\mathbf{v}^{\top},\dots]^{\top}$ as the (positive) right eigenvector of $\widetilde{A}$, where $\mathbf{u}$ and $\mathbf{v}$ are the Perron left and right eigenvectors of the matrix $A$. Thus, we get $\lambda=\lim_{n\to\infty}\left(\widetilde{A}^n)_{ii}\right)^{1/n}$. 

We now define a row-stochastic matrix $\widetilde{P}$ compatible with $\widetilde{A}$ using the matrix $P$ (defined in~\eqref{eq:P-p}), as follows: for each pair $(i_1,j_1), (i_2,j_2)$, let
\[
(\widetilde{P})_{(i_1,j_1),(i_2,j_2)} = \frac{\widetilde{A}_{(i_1,j_1),(i_2,j_2)}\ v_{i_2}}{\lambda\ v_{i_1}}. 
\]
The matrix $\widetilde{P}$ has the following form:
\[
\widetilde{P} = {
	\begin{bmatrix}
		\ddots&\vdots & \vdots & \vdots & \vdots & \vdots  \\
		\dots& P_0 & P_1 & 0 & 0 & \dots \\
		\dots& P_{-1} & P_0 & P_1 & 0 & \dots \\
		\dots &0& P_{-1} & P_0 & P_1 & \dots \\
		\dots &0&0& P_{-1} & P_0  & \dots \\
		\vdots & \vdots & \vdots & \vdots & \vdots & \ddots
\end{bmatrix}},
\]
where $P_{-1}$, $P_{0}$ and $P_1$ are non-zero matrices of size $m$ compatible with $A_{-1}$, $A_{0}$ and $A_1$, respectively and $P_{-1}+P_0+P_1=P$. Observe that $(\widetilde{P})_{(i_1,j_1),(i_2,j_2)} \le P_{i_1,i_2}$.

\subsection{Visualization as a Quasi-Birth-Death Process}

The matrix $\widetilde{A}$ gives an infinite directed graph $\mathcal{G}$ with countably infinitely many vertices labelled by $\{(i,j): i=1,\ldots ,m, \ j\in\mathbb{Z}\}$ whose vertices are grouped into \textit{level sets} indexed by $\mathbb{Z}$ (we call such graph as a two-sided graph). Let us label the collection of $m$ vertices for each fixed $j$ by level $\mathcal{V}_j$, that is, $\mathcal{V}_j=\{(i,j): i=1,\ldots ,m\}$. All the neighbours of vertices in $\mathcal{V}_j$ are contained in $\mathcal{V}_{j-1}\cup \mathcal{V}_{j}\cup \mathcal{V}_{j+1}$. Geometrically, each level $j\in\mathbb{Z}$ corresponds to the fundamental domain $K_j$ and a ``birth'' corresponds to the map $\widetilde{f}$ pushing a region into the right-hand neighboring domain $K_{j+1}$, a ``death'' corresponds to pushing a region into the left-hand neighboring domain $K_{j-1}$, and ``internal transitions'' correspond to mapping within the same domain $K_j$.

We would like to view this directed graph with level sets indexed by the set of integers $\mathbb{Z}$ as a directed graph with level sets indexed by the set of natural numbers $\mathbb{N}$ (we call such a graph as a one-sided graph) to be able to use the theory of ``homogeneous quasi-birth-death process'' (QBD process)~\cite{QBD}. Such a process is a discrete Markov process with two-dimensional state space $\{(i,n)\ : \ i=1,\dots,k,\ n\ge 1\}$. The collection of states $\{(i,n)\ : \ i=1,\dots,k\}$ is known as the $n$-th level of the Markov process. The transition probabilities within each level are independent of the level. More precisely, for $n,n'\ge 1$, the probability $\text{Pr}[X_1=(i,n)\vert X_2=(i',n')]$ depends on $i,i'$ and $n'-n$, but not on the specific values of $n$ and $n'$. Further the transition matrix is block tri-diagonal having the following form:
\[\widetilde{Q}=\begin{bmatrix}
	B & Q_0 & 0 & 0  & \dots   \\
	Q_{2} &  Q_{1} &  Q_{0} & 0  & \dots   \\
	0 &Q_{2} &  Q_{1} &  Q_{0}  & \dots   \\
	\vdots & \vdots & \vdots & \vdots  & \ddots   \\	
\end{bmatrix},
\]
where $Q_0, Q_1, Q_2$, and $B$ are square nonnegative matrices of order $k$ satisfying $Q_0 + Q_1 + Q_2 = Q$ for some row-stochastic matrix $Q$.

At this point, we recall a result~\cite[Theorem 7.3.1]{QBD} adapted to our setting.  
\begin{theorem}\label{thm:qbd-drift}
	Let $Q_j=\begin{bmatrix}
		Q_j^1 & 0 \\ 0& Q_j^2
	\end{bmatrix}$, for $j=0,1,2$, and $Q=\begin{bmatrix}
		Q^1 & 0 \\ 0& Q^2
	\end{bmatrix}$. Suppose that the matrices $Q^1$, $Q^2$ and $\widetilde{Q}$ are all irreducible. Let $\mathbf{q}^j$ be the left Perron eigenvector of $Q^j$, for $j=1,2$. Then the matrix $\widetilde{Q}$ is
	\begin{enumerate}
		\item recurrent if and only if $d_j=\mathbf{q}^j Q_2^j\bar{1} -  \mathbf{q}^j Q_0^j\bar{1}\ge 0$, for $j=1,2$.
		\item positive recurrent if and only if $d_j> 0$, for $j=1,2$.
		\item transient if and only if for at least one of $j=1$ or $j=2$, $d_j< 0$.
	\end{enumerate}
\end{theorem}

The conditions mentioned in the theorem are known as \textit{drift conditions}. We wish to apply the theory of QBD and specifically this result to our system. In this direction, we first transform our graph $\mathcal{G}$ with adjacency matrix $\widetilde{A}$ into a QBD process as follows. For each $n\ge 0$, collect all the vertices in $\mathcal{V}_{-n}\cup \mathcal{V}_{n+1}$ as the $n$-th level. Observe that each level has $2m$ many vertices. The vertices can be re-labelled as $\{(i,n)\ :\ i=1,\dots,2m,\ n\ge 0\}$ so that the adjacency matrix takes the following form. With abuse of notation, we call this new matrix also as $\widetilde{A}$.
\begin{eqnarray}\label{eq:Atilde}
\widetilde{A}&=&\begin{bmatrix}
	C & M_0 & 0 & 0  & \dots   \\
	M_2 &  M_1 &  M_0 & 0  & \dots   \\
	0 & M_2 &  M_1 &  M_0  & \dots   \\
	\vdots & \vdots & \vdots & \vdots  & \ddots   \\
	\end{bmatrix},
\end{eqnarray}
where $C=\begin{bmatrix}
	A_{0} & A_{1} \\ A_{-1}& A_{0}
\end{bmatrix}$, $M_2=\begin{bmatrix}
	A_{1} & 0 \\ 0& A_{-1}
\end{bmatrix}$,  $M_1=\begin{bmatrix}
	A_{0} & 0 \\ 0& A_{0}
\end{bmatrix}$, and $M_0=\begin{bmatrix}
	A_{-1} & 0 \\ 0& A_{1}
\end{bmatrix}$.

With this regrouping and relabelling of vertices, the transition probability matrix takes the following form. With abuse of notation, we call this new matrix also as $\widetilde{P}$.

\begin{eqnarray}\label{eq:qbd}
\widetilde{P}=\begin{bmatrix}
	B & Q_0 & 0 & 0  & \dots   \\
	Q_{2} &  Q_{1} &  Q_{0} & 0  & \dots   \\
	0 &Q_{2} &  Q_{1} &  Q_{0}  & \dots   \\
	\vdots & \vdots & \vdots & \vdots  & \ddots   \\	
\end{bmatrix},
\end{eqnarray}
where $B=\begin{bmatrix}
	P_{0} & P_{1} \\ P_{-1}& P_{0}
\end{bmatrix}$, $Q_2=\begin{bmatrix}
	P_{1} & 0 \\ 0& P_{-1}
\end{bmatrix}$,  $Q_1=\begin{bmatrix}
	P_{0} & 0 \\ 0& P_{0}
\end{bmatrix}$ and $Q_0=\begin{bmatrix}
	P_{-1} & 0 \\ 0& P_{1}
\end{bmatrix}$. Observe that $Q_2 + Q_1 + Q_0 =\begin{bmatrix}
P & 0 \\ 0& P
\end{bmatrix}$. Hence, we have obtained a homogeneous QBD process with transition probability matrix $\widetilde{P}$. Our goal is to study whether the matrix $\widetilde{A}$ is recurrent or transient. Our first step is to establish a connection between the recurrence/transience properties of $\widetilde{A}$ and $\widetilde{P}$. 

\begin{lemma}
The matrix $\widetilde{P}$ is recurrent (transient, null recurrent, positive recurrent) if and only if the matrix $\widetilde{A}$ is recurrent (transient, null recurrent, positive recurrent), respectively.
\end{lemma}

\begin{proof}
	We show that for each $\ell\ge 1$,
	\[
	(\widetilde{P}^\ell)_{(i,n),(i',n')} = \frac{(\widetilde{A}^\ell)_{(i,n),(i',n')}v_i}{\lambda^\ell v_{i'}}.
	\]
	This implies the statement of the lemma. We prove the above identity by induction on $\ell$. The identity holds trivially for $\ell=1$. Suppose it holds for $\ell$, we prove that it holds for $\ell+1$. Observe that
	\begin{align*}
		(\widetilde{P}^{\ell+1})_{(i,n),(i',n')}
		&= \sum_{(i'',n'')} (\widetilde{P}^\ell)_{(i,n),(i'',n'')} \, (\widetilde{P})_{(i'',n''),(i',n')} \\
		&= \sum_{(i'',n'')} \dfrac{(\widetilde{A}^\ell)_{(i,n),(i'',n'')}v_i}{\lambda^\ell v_i''} \, \dfrac{(\widetilde{A})_{(i'',n''),(i',n')}v_i''}{\lambda v_i'}\\
		&= \dfrac{1}{\lambda^{\ell+1}} \sum_{(i'',n'')} (\widetilde{A}^\ell)_{(i,n),(i'',n'')}(\widetilde{A})_{(i'',n''),(i',n')}\dfrac{v_i}{v_i'}\\
		&= \frac{(\widetilde{A}^{\ell+1})_{(i,n),(i',n')}v_i}{\lambda^{\ell+1} v_{i'}}.
	\end{align*}
	Consequently, 
	\[
	(\widetilde{P}^\ell)_{(i,n),(i,n')} = \frac{(\widetilde{A}^\ell)_{(i,n),(i,n')}}{\lambda^\ell}.
	\]
	Hence 
	\[
	\sum_{\ell \geq 0} (\widetilde{P}^\ell)_{(i,n),(i,n)} 
	= \sum_{\ell \geq 0} \frac{(\widetilde{A}^\ell)_{(i,n),(i,n)}}{\lambda^\ell}.
	\]
	\end{proof}

Now we apply the drift conditions to obtain the following result.

\begin{theorem}\label{thm:nonstoc-reducible-case}
    If the matrix $\widetilde{A}$ is irreducible, then it is either null recurrent or transient. 
\end{theorem}
\begin{proof}
Applying Theorem~\ref{thm:qbd-drift} to the transition probability matrix $\widetilde{P}$ in~\eqref{eq:qbd}, we get $d_1=\mathbf{p}P_{1}\mathbf{1}-\mathbf{p}P_{-1}\mathbf{1}$ and $d_2=\mathbf{p}P_{-1}\mathbf{1}-\mathbf{p}P_{1}\mathbf{1}$. These terms are either both zero or have opposite signs. Thus, $\widetilde{P}$ is either null recurrent or transient. The same holds for $\widetilde{A}$ by the preceding lemma.
\end{proof}

\begin{rem}
	We have used tools from the theory of QBD to prove the preceding result since it will allow us to obtain an interesting Theorem~\ref{thm:left-right-drift} about the dynamics of the lifted homeomorphism $\widetilde{f}$. However, the preceding theorem can be proved directly using the tools from the theory of countable Markov chains~\cite{kit}. If the matrix $\widetilde{A}$ is irreducible, it is either transient or recurrent. If it is recurrent, it has a unique pair of left and right Perron eigenvectors which are given by $\widetilde{\mathbf{u}}^{\top}=[\dots, \mathbf{u},\dots,\mathbf{u},\dots]$ and $\widetilde{\mathbf{v}}=[\dots, \mathbf{v},\dots,\mathbf{v},\dots]$, where $\mathbf{u}$ and $\mathbf{v}$ are the left and right Perron eigenvectors of $A$. Since $\widetilde{\mathbf{u}}^{\top}\widetilde{\mathbf{v}}$ is infinite, the matrix $\widetilde{A}$ is null recurrent.
\end{rem}

\noindent In due course, we will in fact establish that if $f$ is a Penner-type pseudo-Anosov map, then the matrix  $\widetilde{A}$ is irreducible, aperiodic and null recurrent.

\subsection{Topological Transitivity and Mixing}\label{subsec:TopTran}

In the setting of a pseudo-Anosov homeomorphism on a surface of finite type, the irreducibility and aperiodicity of the symbolic coding matrix arising from the Markov partition aid in establishing dynamical properties of the homeomorphism, such as topological transitivity, mixing and ergodicity. For lifts of these homeomorphisms on an infinite-type surface, we have obtained a semi-conjugacy from ($\widetilde{\Sigma}_{\widetilde{A}}$, $\sigma$) to $(\mathcal{L}, \widetilde{f})$ in Theorem~\ref{thm:conj}. Thus, it is natural to study the properties of the infinite symbolic coding matrix $\widetilde{A}$. In this section, we prove that if the pseudo-Anosov homeomorphism $f$ on $S_g$ is of Penner-type, then the matrix $\widetilde{A}$ corresponding to the lift $\widetilde{f}$ on $\mathcal{L}$ is both irreducible and aperiodic. 

In Penner's construction of a pseudo-Anosov map, rectangles are constructed along each branch of the train track which give a neighborhood of the train track. By taking the quotient of this neighborhood, the surface $S_g$ is obtained. Let us call this collection of rectangles as $\mathcal{R}_T$. These rectangles are used to construct a Markov partition $\mathcal{R}=\vee_{i=-n}^n f^i(\operatorname{int}(\mathcal{R}_T))$ of $S_g$ for a suitable $n$. Let $A$ be the corresponding transition matrix. This construction is related to a higher block code of the edge shift related to a matrix $M$ which is topologically conjugate to the vertex shift related to matrix $A$. Thus, the matrix $A$ is irreducible and aperiodic if and only if the matrix $M$ is irreducible and aperiodic. We will follow a similar process to show that the transition matrix $\widetilde{A}$ is irreducible and aperiodic. We have the following result.

\begin{theorem}\label{thm:irre}
Let $f$ be a Penner-type pseudo-Anosov map on $S_g$ that factors into a product of liftable Dehn twists under $q_g$. Let $x_0$ be a fixed point of $f$ and let $\widetilde{x}_0\in\mathcal{L}$ be a preimage of $x_0$ under $q_g$. Then the transition matrix $\widetilde{A}$ of the symbolic coding of the lift $\widetilde{f}$ on $\mathcal{L}$ under $q_g$ that fixes $\widetilde{x}_0$ is irreducible and aperiodic.
\end{theorem}

\begin{proof}
	Let $f$ be a Penner-type map on $S_g$ with curve collections $\mathcal{C}\cup\mathcal{D}$, where $\mathcal{C}=\{\gamma_1,\dots,\gamma_n\}$ and $\mathcal{D}=\{\gamma_{n+1},\dots,\gamma_{n+m}\}$, given by $f=T_{\gamma_{r_1}}^{\delta_1}\circ\dots\circ T_{\gamma_s}^{\delta_{s}}$, where $\gamma_{r_i}\in \mathcal{C}\cup\mathcal{D}$ and $\delta_i\in \mathbb{Z}$. \\~\\
	\noindent \textbf{Lifting of curves in $\mathcal{C}\cup\mathcal{D}$ to $\mathcal{L}$}: Since $T_{\gamma}$ is liftable under $q_g$ for each $\gamma\in \mathcal{C}\cup\mathcal{D}$, the preimage of each curve $\gamma\in \mathcal{C}\cup \mathcal{D}$ under $q_g$ is a countable disjoint union of essential simple closed curves in $\mathcal{L}$. For each $\gamma_i\in \mathcal{C}\cup \mathcal{D}$, we can choose an essential simple closed curve on $\mathcal{L}$ from $q_g^{-1}(\gamma_i)$, which we label as $\gamma_{(0,i)}$, such that $\cup_{i=1}^{n+m} \gamma_{(0,i)}$ is a connected set (this is possible since $\mathcal{C}\cup\mathcal{D}$ fills $S_g$).
Let $\mathcal{C}_0=\{\gamma_{(0,1)},\dots,\gamma_{(0,n)}\}$ and $\mathcal{D}_0=\{\gamma_{(0,n+1)},\dots,\gamma_{(0,n+m)}\}$. For each $j\in \mathbb{Z}$, let $\mathcal{C}_j=h_{g-1}^j(\mathcal{C}_0)=\{\gamma_{(j,1)},\dots,\gamma_{(j,n)}\}$ and $\mathcal{D}_j=h_{g-1}^j(\mathcal{D}_0)=\{\gamma_{(j,n+1)},\dots,\gamma_{(j,n+m)}\}$. Without loss of generality, we can assume that $\mathcal{C}_0\cup \mathcal{D}_0$ intersects $\mathcal{C}_j\cup \mathcal{D}_j$ only for $|j|\le 1$ (see Remark~\ref{rem:club}). With notation as above, the lift $\widetilde{f}$ of $f$ is given by \[
	\widetilde{f}=\lim_{k\to\infty}\left(\prod_{j=-k}^k T_{\gamma_{(j,r_1)}}^{\delta_1}\circ\dots\circ \prod_{j=-k}^k T_{\gamma_{(j,r_s)}}^{\delta_s}\right).
	\]  
	\textbf{Construction of a neighbourhood around $\mathcal{C}\cup\mathcal{D}$}: Let $N$ be an annular neighbourhood of curves from $\mathcal{C} \cup \mathcal{D}$. The set $N$ is connected since $\mathcal{C} \cup \mathcal{D}$ fills $S_g$ (we can choose $N$ to be the train track neighbourhood in Penner's construction). Let $\widetilde{N}_0$ be a connected neighbourhood of curves from $\cup_{\gamma\in \mathcal{C}_0 \cup \mathcal{D}_0}\gamma$, that is a subset of the preimage $q_g^{-1}(N)$ in $\mathcal{L}$ given by the union of annular neighbourhoods of curves from collection of curves $\mathcal{C}_0 \cup \mathcal{D}_0$. Moreover, $q_g(\widetilde{N}_0)=N$. For each $j\in\mathbb{Z}$, let $\widetilde{N}_j=h_{g-1}^j(\widetilde{N}_0)$, that is a neighbourhood of the union of curves from $\mathcal{C}_j\cup\mathcal{D}_j$, that is an annular neighbourhood of curves from $\mathcal{C}_j\cup\mathcal{D}_j$. Observe that $q_g(\widetilde{N}_j)=N$ and note that $\widetilde{N}_j$ intersects $\widetilde{N}_{j'}$ only when $|j-j'|\le 1$. Consequently, $\widetilde{N}=\cup_{j\in\mathbb{Z}}\widetilde{N}_j$ is a neighbourhood of curves from $\widetilde{\mathcal{C}}\cup\widetilde{\mathcal{D}}$. Within $\widetilde{N}_0$, the collection $\mathcal{C}_0 \cup \mathcal{D}_0$ satisfies the following:
	\begin{itemize}
		\item Each curve in the collection is essential and simple.
		\item The collection $\mathcal{C}_0$ hits $\mathcal{D}_0$ efficiently, mirroring the efficiency of $\mathcal{C}$ and $\mathcal{D}$ in $S_g$.
		\item The collection $\mathcal{C}_0 \cup \mathcal{D}_0$ fills $\widetilde{N}_0$.
	\end{itemize}
	Since $f$ is a (Penner-type) pseudo-Anosov map on $S_g$, we get that $f_0=T_{\gamma_{(0,r_1)}}^{\delta_1}\circ\dots\circ T_{\gamma_{(0,r_s)}}^{\delta_s}$ is a (Penner-type) pseudo-Anosov map on the annular neighbourhood $\widetilde{N}_0$. \\~\\
	\noindent \textbf{Incidence intersection matrix}: Let $\widetilde{\Omega}_0$ be the intersection matrix of $\mathcal{C}_0\cup\mathcal{D}_0$ in $\widetilde{N}_0$ (a surface of finite type) and let $\Omega$ be the intersection matrix of $\mathcal{C}\cup\mathcal{D}$. With notation as in Theorem~\ref{thm:penner}, the incidence intersection matrix of $f_0$ is given by \[M_{f_0}=B_{(0,r_1)}^{|\delta_1|}\dots B_{(0,r_s)}^{|\delta_s|}.\]
	Since $f_0$ is a pseudo-Anosov map on $N_0$, the matrix $M_{f_0}$ is irreducible and aperiodic. Also, the incidence intersection matrix of $f$ is given by \[M_{f}=B_{r_1}^{|\delta_1|}\dots B_{r_s}^{|\delta_s|}.\]
	Since $\widetilde{\Omega}_0\leq\Omega$ entry-wise, we have $M_{f_0}\leq M_f$. 
	
	We can visualize $f_0$ as a map on $\mathcal{L}$ (or, on the neighbourhood $\widetilde{N}$). Let $\widetilde{\Omega}$ be the intersection matrix of the collection of curves $\widetilde{\mathcal{C}}\cup\widetilde{\mathcal{D}}$ on $\mathcal{L}$, whose entries we will arrange as follows. Enumerate its rows and columns using the set $\mathcal{I}_1=\{(j,i) \: \ i=1,\dots,n+m, \ j\in\mathbb{Z}\}$ arranged in the dictionary order, to get
	\[  \widetilde{\Omega}=\begin{bmatrix}
		\ddots  &   \vdots &  \vdots & \vdots & \vdots & \vdots & \vdots \\
		\dots & \Omega_1 & \Omega_2 & 0 & 0 & 0 &\dots \\
		\dots &  \Omega_0 & \Omega_1 & \Omega_2 & 0 & 0 & \dots \\
		\dots &   0 & \Omega_0 & \fbox{$\Omega_1$} & \Omega_2 & 0 &\dots \\
		\dots &   0 & 0 & \Omega_0 & \Omega_1 & 0 &\dots \\
		\dots &  0 & 0 & 0 &\Omega_0 & \Omega_1 &  \dots \\
		\vdots &   \vdots &\vdots & \vdots & \vdots & \vdots & \ddots 
	\end{bmatrix},\] 
	with $\Omega_1=\widetilde{\Omega}_0$ and $\Omega_2=\Omega_0^{\top}$ (using $i(\gamma,\beta)=i(\beta,\gamma)$). Note that the center block $\Omega_0$ corresponds to the rows and columns enumerated by $\{(0,i) \ : \ i=1,\dots,n+m\}$. The matrices $\Omega_1, \Omega_2$ are nonzero matrices since $\cup_{\gamma\in\widetilde{\mathcal{C}}\cup\widetilde{\mathcal{D}}}\ \gamma$ is a connected set. Moreover $\Omega_0+\Omega_1+\Omega_2=\Omega$. 
	
	Along the lines of Theorem~\ref{thm:penner}, for the map $f_0$ on $\mathcal{L}$, we define the matrix $\widetilde{B}_{(k,l)}$ using $\widetilde{\Omega}$ for each $\gamma_{(k,l)}\in \widetilde{\mathcal{C}}\cup\widetilde{\mathcal{D}}$ as follows. Let $\widetilde{B}_{(k,l)}=I+\widetilde{E}_{(k,l)}\widetilde{\Omega}$, where $I$ is the identity matrix whose rows and columns are indexed by the set $\mathcal{I}_1$ and $\widetilde{E}_{(k,l)}$ is the matrix whose rows and columns are indexed by the set $\mathcal{I}_1$ having all entries zero except for the $(k,l)^{th}$ diagonal entry, which is 1. 
	
	Define $\widetilde{M}_{f_0}=\widetilde{B}_{(0,r_1)}^{|\delta_1|}\dots \widetilde{B}_{(0,r_s)}^{|\delta_s|}$, which is a well-defined matrix being a finite product of locally-finite matrices. It has the following structure:
     \[
	\widetilde{M}_{f_0}=\begin{bmatrix}
		\ddots &   \vdots & \vdots & \vdots & \vdots \\
		\dots &  I &0  &0& \dots \\
		\dots &   M' & \fbox{$M_{f_0}$} & M'' & \dots \\
		\dots &   0 & 0 & I & \dots \\
		\vdots &  \vdots & \vdots & \vdots & \ddots
	\end{bmatrix},
	\]
	 where $I$ is the identity matrix of size $n+m$. Note that we may have to follow a procedure similar to Remark~\ref{rem:club} to bring the matrix in the above form. Now for each $k\in \mathbb{Z}$, we follow the similar process for the map \[f_k=\prod_{j=-k}^k T_{\gamma_{(j,r_1)}}^{\delta_1}\circ\dots\circ \prod_{j=-k}^k T_{\gamma_{(j,r_s)}}^{\delta_s}\] on $\cup_{j=-k}^k \widetilde{N}_j$, to obtain the incidence intersection matrix \[\widetilde{M}_{f_k}=\left(\prod_{j=-k}^k\widetilde{B}_{(j,r_1)}^{|\delta_1|}\right)\dots \left(\prod_{j=-k}^k\widetilde{B}_{(j,r_s)}^{|\delta_s|}\right).\]
	 
	For each $(j,i)\in \mathcal{I}_1$, the sequence $\{(\widetilde{M}_{f_k})_{(j,i)}\}_{k\ge 1}$ of the the $(j,i)$-th entries of the sequence of matrices $\{\widetilde{M}_{f_k}\}_{k\ge 1}$ is eventually constant as $f_k\to \widetilde{f}$ in $\mathrm{Map}(\mathcal{L})$. Thus, the entry-wise limit of the sequence $\{\widetilde{M}_{f_k}\}_{k\ge 1}$ of matrices, denoted as $M_{\widetilde{f}}$, is given by 
	\[\begin{bmatrix}
		\ddots  &   \vdots &  \vdots & \vdots & \vdots & \vdots & \vdots \\
		\dots & M_1 & M_2 & 0 & 0 & 0 &\dots \\
		\dots &  M_0 & M_1 & M_2 & 0 & 0 & \dots \\
		\dots &   0 & M_0 & \fbox{$M_1$} & M_2 & 0 &\dots \\
		\dots &   0 & 0 & M_0 & M_1 & 0 &\dots \\
		\dots &  0 & 0 & 0 &M_0 & M_1 &  \dots \\
		\vdots &   \vdots &\vdots & \vdots & \vdots & \vdots & \ddots 
	\end{bmatrix},\]
	where $M_1=M_{f_0}$ (which is irreducible and aperiodic). Note that the matrix $M_{\widetilde{f}}$ has the above structure due to the structure of $\widetilde{\Omega}$. Since $\Omega_2=\Omega_1^{\top}$ is a nonzero matrix and the Dehn twist along each curve in the collection appears in $\widetilde{f}$, the matrices $M_0$ and $M_2$ are nonzero matrices. Hence, $M_{\widetilde{f}}$ is irreducible and aperiodic.\\~\\
	\noindent \textbf{Train track measure}: Let $\{b_1,\ldots,b_{m'}\}$ be the branches of the train track $\tau$ constructed using $\mathcal{C}\cup\mathcal{D}$ on $S_g$. Consider $\widetilde{\tau}=q_g^{-1}(\tau)$ with branches $\{\widetilde{b}_{(\alpha,\beta)} \ : \ (\alpha,\beta)\in\mathcal{I}_2\}$ enumerated by the countable set $\mathcal{I}_2=\{(\alpha,\beta)\ : \ \beta=1,\dots,m', \text{ and } \alpha\in \mathbb{Z}\}$. The labeling of the pre-image branches is such that the branches in $\{\widetilde{b}_{(\alpha,\beta)} \ : \ \beta=1,\dots,m'\}$ correspond to the curve collection $\mathcal{C}_\alpha\cup \mathcal{D}_\alpha$ on $\mathcal{L}$. Let $t$ be a homeomorphism on $\mathcal{L}$ homotopic to the identity satisfying $\widetilde{\tau}\subset t(\widetilde{\mathcal{C}}\cup\widetilde{\mathcal{D}})$. 
	
	Let $W(\widetilde{\tau})$ denote the space of train-track weights on $\widetilde{\tau}$. We will now describe how the train-track transition matrix $\widetilde{M}$ of $\widetilde{f}$ acts on the weights of the train track $\widetilde{\tau}$. 
	\begin{itemize}
		\item Each element of $W(\widetilde{\tau})$ can be viewed as an element of $\mathbb{R}^{\mathcal{I}_2}$. Let us consider the following measure in $W(\widetilde{\tau})$. For each curve $\gamma_{(j,i)}\in \widetilde{\mathcal{C}}\cup \widetilde{\mathcal{D}}$, consider the measure $ \widetilde{p}_{(j,i)}$ defined as
		\[
		\widetilde{p}_{(j,i)}(\widetilde{b})=
		\begin{cases}
			1,&\text{if }\widetilde{b}\subset t({\gamma}_{(j,i)}),\\
			0,&\text{otherwise},
		\end{cases}
		\] 
		for each $\widetilde{b}\in \widetilde{\tau}$. Each $\widetilde{p}_{(j,i)}$ can be seen as an elementary vector in $\mathbb{R}^{\mathcal{I}_2}$. Following Penner, let $\widetilde{H}={\operatorname{span}}\{\widetilde{p}_{(j,i)} \ : \ \gamma_{(j,i)}\in\widetilde{\mathcal{C}}\cup\widetilde{\mathcal{D}}\}$ be the cone generated by these measures. 
\item Suppose $\widetilde{f}(\widetilde{\tau})$ is carried by $\widetilde{\tau}$, that is, there exists a carrying map $\rho:\widetilde{f}(\widetilde{\tau})\longrightarrow\widetilde{\tau}$, obtained by isotoping $\widetilde{f}(\widetilde{\tau})$ into a regular neighborhood $N(\widetilde{\tau})$ of $\widetilde{\tau}$ and collapsing each tie of $N(\widetilde{\tau})$ to a point. Then $\widetilde{f}$ induces a linear map $\widetilde{f}_*:W(\widetilde{\tau})\longrightarrow W(\widetilde{\tau}),$
		called the induced action of $\widetilde{f}$ on train-track weights defined as follows. Consider a train-track weight $\mu\in W(\tau)$ regarded as a transverse measure on $\tau$. The image $\widetilde{f}(\tau)$ then inherits the push-forward measure $\widetilde{f}_*\mu$ on $\widetilde{f}(\widetilde{\tau})$. Applying the carrying map $\rho$ yields a transverse measure on $\widetilde{\tau}$, which we again denote as $\widetilde{f}_*\mu$ (abuse of notation). Equivalently, if $b$ is a branch in $\widetilde{\tau}$, then the weight assigned to a branch $b$ is given by
		\[
		(\widetilde{f}_*\mu)(b)
		=
		\sum_{b':\,\rho(\widetilde{f}(b'))\supset b}
		n_{b'b}\,\mu(b'),
		\]
		where $n_{b'b}$ is the number of times the path $\rho(\widetilde{f}(b'))$ traverses the branch $b$. Note that this sum is finite since the train track is locally finite and $\widetilde{f}$ take compact set to compact set. Consequently, 
		\[
		\widetilde{f}_*(\mu)=\widetilde{M}\mu,
		\]
		for every $\mu\in W(\widetilde{\tau})$, where recall that $\widetilde{M}$ is the train-track transition matrix of $\widetilde{f}$.
	\end{itemize}
	
	Similarly, we have the induced action of Dehn twist $T_{{\gamma}_{(k,l)}}$ (or any homeomorphism) on $W(\widetilde{\tau})$, which will be given by matrix, say $\widetilde{M}_{(k,l)}$. For each Dehn twist $T_{{\gamma}_{(k,l)}}$, the induced action on the basis
    element $\widetilde{p}_{(j,i)}$ is given by
    \[
    (T_{{\gamma}_{(k,l)}})_*(\widetilde{p}_{(j,i)})
    =
    \widetilde{p}_{(j,i)}
    +
    i(\gamma_{(k,l)},\gamma_{(j,i)})
    \,\widetilde{p}_{(k,l)},
    \]
    where $i(\gamma_{(k,l)},\gamma_{(j,i)})$ denotes the geometric intersection number of the curves
    $\gamma_{(k,l)}$ and $\gamma_{(j,i)}$. Let
    \[
    \widetilde{C}:\mathbb{R}^{\mathcal{I}_1}\longrightarrow W(\widetilde{\tau})\subset\mathbb{R}^{\mathcal{I}_2}
    \] be the matrix whose the $(j,i)$-th column is the vector $\widetilde{p}_{(j,i)}$.  Note that $\widetilde{C}(\mathbb{R}^{\mathcal{I}_1})=\widetilde{H}$ and since $\widetilde{C}$ is injective, it identifies the cone generated by the measures $\widetilde{p}_{(j,i)}$. From the definition of $\widetilde{B}_{(k,l)}$, we see that the induced action of $T_{\gamma_{(k,l)}}$ on $\widetilde{H}$ is given by the matrix multiplication $\widetilde{C}\widetilde{B}_{(k,l)}$ (well-defined). Thus, for the lift $\widetilde{f}$, the induced action on $\widetilde{H}$ is represented by $\widetilde{C}M_{\widetilde{f}}$. Note that $\widetilde{H}$ is an invariant subcone of $W(\widetilde{\tau})$ under $\widetilde{M}$, that is, $\widetilde{M}(\widetilde{H})\subset \widetilde{H}$.\\~\\
    \noindent \textbf{Relation between $M_{\widetilde{f}}$ and $\widetilde{M}$}: For each branch $\widetilde{b}_{(\alpha,\beta)}$, define $\mu_{(\alpha,\beta)}\in W(\widetilde{\tau})$ as
    \[
    \mu_{(\alpha,\beta)}(\widetilde{b})=
    \begin{cases}
    1,&\text{if }\widetilde{b}=\widetilde{b}_{(\alpha,\beta)},\\
    0,&\text{otherwise}.
    \end{cases}
    \]
  Thus
    \[
    \widetilde{p}_{(j,i)}
    =
    \sum_{\widetilde{b}_{(\alpha,\beta)}
    \subset
    t(\gamma_{(j,i)})}
    \mu_{(\alpha,\beta)}.
    \]
    For the standard unit vector $e_{(j,i)}\in\mathbb{R}^{\mathcal{I}_1},$
    we have
    \[
    \widetilde{C}\widetilde{B}_{(k,l)}(e_{(j,i)})
    =
    \widetilde{C}
    \left(
    e_{(j,i)}
    +
    i(\gamma_{(k,l)},\gamma_{(j,i)})e_{(k,l)}
    \right)
    =
    \widetilde{p}_{(j,i)}
    +
    i(\gamma_{(k,l)},\gamma_{(j,i)})
    \widetilde{p}_{(k,l)}.
    \]
    Thus, the action of a Dehn twist $T_{\gamma_{(k,l)}}$ with respect to the basis $\{\widetilde{p}_{(j,i)}\}$ is represented by $\widetilde{B}_{(k,l)}$, while its action in branch coordinates is represented by the corresponding train-track transition matrix $\widetilde{M}_{(k,l)}$. Now consider
    \[
    \widetilde{M}_{(k,l)}
    \widetilde{C}(e_{(j,i)})
    =
    \widetilde{M}_{(k,l)}(\widetilde{p}_{(j,i)})
    =
    \widetilde{M}_{(k,l)}
    \left(
    \sum_{\widetilde{b}_{(\alpha,\beta)}
    \subset
    t(\gamma_{(j,i)})}
    \mu_{(\alpha,\beta)}
    \right)
    =
    \sum_{\widetilde{b}_{(\alpha,\beta)}
    \subset
    t(\gamma_{(j,i)})}
    \widetilde{M}_{(k,l)}(\mu_{(\alpha,\beta)}).
    \]
    Now note that
    \[
    \widetilde{M}_{(k,l)}(\mu_{(\alpha,\beta)})
    =
    \mu_{(\alpha,\beta)}
    +
    \sum_{\widetilde{b}_{(\alpha',\beta')}
    \subset
    t(\gamma_{(k,l)})}
    \mu_{(\alpha',\beta')},
    \]
    if $\widetilde{b}_{(\alpha,\beta)}$ intersects the neighbourhood of $\gamma_{(k,l)}$ in which Dehn twist $T_{\gamma_{(k,l)}}$ is applied, and
    \[
    \widetilde{M}_{(k,l)}(\mu_{(\alpha,\beta)})
    =
    \mu_{(\alpha,\beta)},
    \]
    otherwise. Therefore,
    \[
    \sum_{\widetilde{b}_{(\alpha,\beta)}
    \subset
    t(\gamma_{(j,i)})}
    \widetilde{M}_{(k,l)}(\mu_{(\alpha,\beta)})
    =
    \sum_{\widetilde{b}_{(\alpha,\beta)}
    \subset
    t(\gamma_{(j,i)})}
    \mu_{(\alpha,\beta)}
    +
    \sum_{\substack{
    \widetilde{b}_{(\alpha,\beta)}
    \subset
    t(\gamma_{(j,i)})\\
    \widetilde{b}_{(\alpha,\beta)}
    \cap
    N_{t(\gamma_{(k,l)})}
    \neq
    \emptyset
    }}
    \;
    \sum_{\widetilde{b}_{(\alpha',\beta')}
    \subset
    t(\gamma_{(k,l)})}
    \mu_{(\alpha',\beta')}.
    \]
    Here,
    $N_{t(\gamma_{(k,l)})}$ denotes a neighbourhood of
    $t(\gamma_{(k,l)})$
    on which the Dehn twist
    $T_{\gamma_{(k,l)}}$
    is supported. Hence,
    \[
    \begin{aligned}
    \sum_{\widetilde{b}_{(\alpha,\beta)}
    \subset
    t(\gamma_{(j,i)})}
    \widetilde{M}_{(k,l)}(\mu_{(\alpha,\beta)})
    &=
    \widetilde{p}_{(j,i)}
    +
    \sum_{\substack{
    \widetilde{b}_{(\alpha,\beta)}
    \subset
    t(\gamma_{(j,i)})\\
    \widetilde{b}_{(\alpha,\beta)}
    \cap
    N_{t(\gamma_{(k,l)})}
    \neq
    \emptyset
    }}
    \widetilde{p}_{(k,l)}\\
    &=
    \widetilde{p}_{(j,i)}
    +
    i(\gamma_{(k,l)},\gamma_{(j,i)})
    \widetilde{p}_{(k,l)}.
    \end{aligned}
    \]
    Therefore,
    \[
    \widetilde{M}_{(k,l)}\widetilde{C}
    =
    \widetilde{C}\widetilde{B}_{(k,l)}.
    \]
    Consequently we have (for $\widetilde{f}$),
    \[
    \widetilde{M}\widetilde{C}
    =
    \widetilde{C}M_{\widetilde{f}}.
    \]

    \noindent \textbf{$\widetilde{M}$ is irreducible and aperiodic}: Consider the map $f^2$ on $S_g$ such that the Dehn twist $T_{\gamma}$ in $f$ first twists the curves $\bar{t}(\gamma)^+$ and then twists the curve $\bar{t}(\gamma)^-$, where $\bar{t}$ be a homeomorphism on $S_g$ homotopic to the identity satisfying ${\tau}\subset \bar{t}({\mathcal{C}}\cup{\mathcal{D}})$. Consider the corresponding lift $\widetilde{f}^2$ on $\mathcal{L}$ by taking the preimage of neighborhoods under $q_g$ (see Figure~\ref{fig:both_side_nbhd}). For $(\alpha,\beta),(\alpha',\beta')\in \mathcal{I}_2$, consider $\mu_{(\alpha,\beta)},\mu_{(\alpha',\beta')}\in W(\widetilde{\tau})$ and $\widetilde{p}_{(j,i)},\widetilde{p}_{(j',i')}\in \operatorname{Im}(\widetilde{C})$ such that  $\widetilde{b}_{(\alpha',\beta')}\subset t(\gamma_{(j',i')})$ and $\widetilde{b}_{(\alpha,\beta)}\cap N_{t(\gamma_{(j,i)})}\neq \emptyset$. Such a neighbourhood $N_{t(\gamma_{(j,i)})}$ exists since we apply $\widetilde{f}^2$ on both sides of the curve $t(\gamma_{(j,i)})$. Thus, $\widetilde{f}^2(\widetilde{b}_{(\alpha,\beta)})$ wraps around $t(\gamma_{(j,i)})$ since the local picture of the lift agrees with the Penner's construction on the base surface $S_g$. Hence, $\widetilde{M}^2(\mu_{(\alpha,\beta)})=\widetilde{p}_{(j,i)}+\text{extra term}$. 
    
    Since $M_{\widetilde{f}}$ is a non-negative irreducible and aperiodic matrix, we have that ${M}^2_{\widetilde{f}}={M}_{\widetilde{f}^2}$ is a non-negative irreducible and aperiodic matrix. Thus, for any $(j,i),(j',i')\in \mathcal{I}_1$, there exists a natural number $k$ such that $(\widetilde{M}^2_{\widetilde{f}})^k_{(j,i)(j',i')}>0$. Consider 
    \[\widetilde{M}^{2k+2}(\mu_{(\alpha,\beta)})=\widetilde{M}^{2k}\widetilde{M}^{2}(\mu_{(\alpha,\beta)})=\widetilde{M}^{2k}(\widetilde{p}_{(j,i)})+\widetilde{M}^{2k}(\text{extra term}).\] 
    Further note that 
        \[
    \begin{aligned}
    \widetilde{M}^{2k}(\widetilde{p}_{(j,i)})
    &=\widetilde{M}^{2k}\widetilde{C}(e_{(j,i)}) =\widetilde{C}(M_{\widetilde{f}}^{2k})(e_{(j,i)}) =\widetilde{C}\left(
    \sum_{(j'',i'')\in\mathcal{I}_1}
    \left(M_{\widetilde{f}}^{2k}\right)_{(j,i)(j'',i'')}
    e_{(j'',i'')}
    \right) \\
    &=\sum_{(j'',i'')\in\mathcal{I}_1}
    \left(M_{\widetilde{f}}^{2k}\right)_{(j,i)(j'',i'')}
    \widetilde{p}_{(j'',i'')}.
    \end{aligned}
    \] 
    Since $\widetilde{p}_{(j',i')}=\sum_{\widetilde{b}_{(\alpha,\beta)}\subset t(\gamma_{(j',i')})}\mu_{(\alpha,\beta)}$ and $\widetilde{b}_{(\alpha',\beta')}\subset t(\gamma_{(j',i')})$, we have $\widetilde{M}^{2k}(\widetilde{p}_{(j,i)})$ contains $\mu_{(\alpha',\beta')}$ with positive coefficient $(\widetilde{M}^2_{\widetilde{f}})^k_{(j,i)(j',i')}>0$.
    
    Thus, $\widetilde{M}^{2k+2}(\mu_{(\alpha,\beta)})$ contains $\mu_{(\alpha',\beta')}$ with positive coefficient ($\geq (\widetilde{M}^2_{\widetilde{f}})^k_{(j,i)(j',i')}>0$). Hence, $(\widetilde{M}^{2k+2})_{(\alpha,\beta)(\alpha',\beta')}>0$, which implies that $\widetilde{M}$ is irreducible. Moreover, $\widetilde{M}$ has non-zero diagonal entries due to $\widetilde{f}$ being a lift of $f$ from the Penner's construction and the definition of $\widetilde{M}$. Thus, the period of the vertex $(\alpha,\beta)$ in definition of aperiodic matrix given by $\operatorname{gcd}\{k>0 \ : \  (\widetilde{M}^{k})_{(\alpha,\beta)(\alpha,\beta)}>0\}$ equals 1. Since $\widetilde{M}$ is irreducible, it follows that $\widetilde{M}$ is aperiodic.\\
    \begin{figure}[h!]
	\centering
		\begin{tikzpicture}
				
				\node[anchor=south west, inner sep=0] (image) at (0,0) {\includegraphics[width=0.5\linewidth]{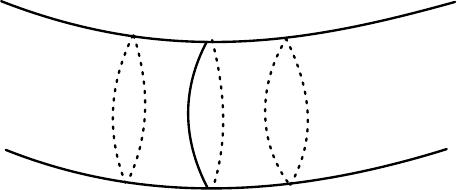}};
				
				\begin{scope}[x={(image.south east)}, y={(image.north west)}]

					\node[label=below:$t(\gamma_{(j,i)}^+)$] at (0.3,0.07) {};
                    \node[label=below:$t(\gamma_{(j,i)}^-)$] at (0.7,0.07) {};
					\node[label=below:$t(\gamma_{(j,i)})$] at (0.48,0.05) {};
                    \node[label=below:$N_{t(\gamma_{(j,i)})} $] at (0.15,0.5) {};
				\end{scope}
			\end{tikzpicture}
			\caption{A neighborhood in which the Dehn twists are supported.}
			\label{fig:both_side_nbhd}
	\end{figure}

    \noindent \textbf{Final Step -- Connecting $\widetilde{A}$ and $\widetilde{M}$}: Recall that to construct the Markov partition $\widetilde{\mathcal{R}}$ of $\mathcal{L}$, we lift the Markov partition $\mathcal{R}$ of $S_g$. We can construct $\mathcal{R}$ by using the collection of rectangles $\mathcal{R}_{\tau}$ (rectangles around each branch) as ${\mathcal{R}}=\vee_{i=-n}^n {f}^i(\operatorname{int}(\mathcal{R}_{{\tau}}))$, for a suitable $n$ (see~\cite{flp}). We can directly take $\widetilde{\mathcal{R}}=\vee_{i=-n}^n \widetilde{f}^i(\operatorname{int}(\mathcal{R}_{\widetilde{\tau}}))$ using the collection of rectangles $\mathcal{R}_{\widetilde{\tau}}$ along the train track $\widetilde{\tau}\subset \mathcal{L}$. Thus, we get the transition matrix $\widetilde{A}$ as in~\eqref{eq:Atilde}. Finally, we obtain a higher block code of the edge shift related to matrix $\widetilde{M}$ which is topologically conjugate to the shift defined by the matrix $\widetilde{A}$. Since the matrix $\widetilde{M}$ is irreducible and aperiodic, we get that the matrix $\widetilde{A}$ is irreducible and aperiodic. This concludes the proof.
\end{proof}

Using the above lemma, we prove that the Markov shift $(\widetilde{\Sigma}_{\widetilde{A}},\sigma)$ is both topologically transitive and  mixing. Through $\widetilde{\pi}$, we get that the lift  $(\mathcal{L},\widetilde{f})$ is also both topologically transitive and mixing. We use the following results from the literature, which are given as Observations 7.2.1. and 7.2.2. in~\cite[Chapter 7]{kit}.
 
 \begin{theorem}
 	Let $B$ be a $0-1$ matrix with countable infinite index set. Then if the matrix $B$ is
 	\begin{enumerate}[(i)]
 		\item irreducible, then the dynamical system $(\Sigma_B, \sigma)$ is topologically transitive.
 		\item both irreducible and aperiodic, then the dynamical system $(\Sigma_B, \sigma)$ is topologically mixing.
 	\end{enumerate}
\end{theorem}
         
Using this result, Theorem~\ref{thm:irre} and the symbolic coding map $\widetilde{\pi}$, we obtain the following result.

\begin{theorem}\label{thm:TopoTranandmixingth}
For $\widetilde{f}$, the dynamical system $(\mathcal{L}, \widetilde{f})$ is topologically transitive and mixing.
\end{theorem}

Note that the topological transitive property of $(\mathcal{L}, \widetilde{f})$ implies that no subsurface of $\mathcal{L}$ is invariant under the map $\widetilde{f}$. 
        
\subsection{\texorpdfstring{Ergodicity and null recurrence of $\widetilde{f}$}{Ergodicity and null recurrence of f~}}\label{subsec:ergnull}
We have just established that the matrix $\widetilde{A}$ is irreducible and aperiodic. By Theorem~\ref{thm:nonstoc-reducible-case}, it is either null recurrent or transient. In this section, we will explore this further. 

Let $(X,\mathcal{B},\nu)$ be measure space. For a measurable set $B\in\mathcal{B}$, define the first return map $\tau_B: B \to \mathbb{N}\cup\{\infty\}$ as follows:
\[\tau_B(x)=\inf\{n\geq 1:g^n(x)\in B\}.\] 

\begin{definition}
	We say the set $B$ is \textit{recurrent} if $\tau_B<\infty$, a.e. $x\in B$. The map $g$ is \textit{recurrent} if for every set $B\in \mathcal{B}$ is recurrent, in which case we say that $(X,\mathcal{B},\nu,g)$ is a recurrent system. Otherwise, it is known as \textit{transient}.
\end{definition}

\noindent For a recurrent set $B\in\mathcal{B}$ with $\nu(B)<\infty$,
\begin{eqnarray}\label{eq:Brec}
\nu\left(\bigsqcup_{n>0}\left(B\cap g^{-1}(B^c)\cap\dots\cap g^{-(n-1)}(B^c)\cap g^{-n}(B)\right)\right)&=&\nu(B),
\end{eqnarray} 
that is, \[\lim_{n\to\infty}\nu(B\cap g^{-1}(B^c)\cap\dots\cap g^{-(n-1)}(B^c)\cap g^{-n}(B))=0.\]

\begin{definition}
	We say that a recurrent system $(X,\mathcal{B},\nu,g)$ is \textit{positive recurrent} if every set of finite measure $B\in\mathcal{B}$, $\sum_{n\geq 1}n\ \nu(\tau_B=n)<\infty$. \\
	We say a recurrent system $(X,\mathcal{B},\nu,g)$ is \textit{null recurrent} if there exists a set of finite measure $B\in \mathcal{B}$ for which $\sum_{n\geq1} n \ \nu(\tau_B=n)=\infty$.
\end{definition}

When $\nu$ is a finite measure, by the Poincar\'e-recurrence theorem and Kac's lemma, the system $(X,\mathcal{B},\nu)$ is recurrent and positive recurrent. Hence, the dynamical system $(S_g,f,\mu=\mu_A\circ\pi^{-1})$ is recurrent and positive recurrent (this does not use the fact that $f$ is a pseudo-Anosov map).

Now consider the lift $\widetilde{f}$ and the corresponding matrix of symbolic coding $\widetilde{A}$ which is indexed by a countable infinite set. Consider the $\sigma$-algebra $\widetilde{\mathcal{B}}_{\widetilde{A}}$ generated by the cylinder sets of $\widetilde{\Sigma}_{\widetilde{A}}$ and the Markov measure $\mu_{\widetilde{A}}$ on $\widetilde{\mathcal{B}}_{\widetilde{A}}$ given by the row-stochastic matrix $\widetilde{P}$ (as defined earlier). We thus obtain a measure space $(\widetilde{\Sigma}_{\widetilde{A}},\widetilde{\mathcal{B}}_{\widetilde{A}},\mu_{\widetilde{A}})$ with a $\sigma$-finite measure $\mu_{\widetilde{A}}$. The shift map $\sigma$ is a measure-preserving transformation on this space. 

Since  $(\widetilde{\Sigma}_{\widetilde{A}},\widetilde{\mathcal{B}}_{\widetilde{A}},\mu_{\widetilde{A}})$ is a $\sigma$-finite measure space, we cannot use either the Poincar\'e recurrence theorem or Kac's lemma. We will use techniques from infinite ergodic theory~\cite{inferg} to analyze the recurrence properties of this system. Further we use this to study the dynamical properties of the system $(\mathcal{L},\widetilde{f})$ for the measure $\widetilde{\mu}:=\mu_{\widetilde{A}}\circ \widetilde{\pi}^{-1}$. 
 
By checking on the cylinder sets, we can easily see that map $\sigma$ on $(\widetilde{\Sigma}_{\widetilde{A}},\widetilde{\mathcal{B}}_{\widetilde{A}},\mu_{\widetilde{A}})$ is recurrent if and only if the matrix $\widetilde{A}$ is recurrent. Similarly, if the matrix $\widetilde{A}$ is recurrent then the map $\sigma$ on $(\widetilde{\Sigma}_{\widetilde{A}},\widetilde{\mathcal{B}}_{\widetilde{A}},\mu_{\widetilde{A}})$ is positive (null) recurrent if and only if the matrix $\widetilde{A}$ is positive (null) recurrent. We have the following assertions which can be proved using standard arguments.

\begin{theorem}
	The following statements hold true:
	\begin{enumerate}[(i)]
		\item The Markov measure $\mu_{\widetilde{A}}$ is non-atomic.
		\item The shift map $\sigma$ is measure-preserving with respect to the measure $\mu_{\widetilde{A}}$.
		\item The shift map $\sigma$ is ergodic with respect to the measure $\mu_{\widetilde{A}}$.
	\end{enumerate}
\end{theorem}

\begin{proof}
	Consider $\bar{v}=\{(i_n,j_n)\}_{n\in\mathbb{Z}}\in \Sigma_{\widetilde{A}}$. Then $\{(i_n)\}_{n\in\mathbb{Z}}\in \Sigma_{A}$ and by the definition of $\widetilde{P}$ and $\mu_{\widetilde{A}}$, we have $\mu_{\widetilde{A}}(\bar{v})\le \mu_A(\{(i_n)\}_{n\in\mathbb{Z}})$. Since $\mu_A$ is a non-atomic measure, the measure $\mu_{\widetilde{A}}$ is also non-atomic.

	Let $[v_0\dots v_\ell]$ be a cylinder set in $\mathcal{B}_{\widetilde{A}}$. As $\widetilde{A}$ has a block structure, there are only finitely many indices $u_1,\dots,u_s$ such that $\widetilde{A}_{u_i,v_0}=1$. Hence,  
	\begin{eqnarray*}
		\mu_{\widetilde{A}}(\sigma^{-1}([v_0\dots v_\ell]))&=&\mu_{\widetilde{A}}(\bigsqcup_{i=1}^{s}([u_i,v_0\dots v_\ell]))=\sum_{i=1}^{s}\mu_{\widetilde{A}}([u_i,v_0\dots v_\ell])=\sum_{i=1}^{s}\widetilde{p}_{u_i}\widetilde{P}_{u_i v_0}\dots \widetilde{P}_{v_{\ell-1} v_\ell}\\
		&=&\widetilde{p}_{v_0}\widetilde{P}_{v_0 v_1}\dots \widetilde{P}_{v_{\ell-1} v_\ell}=\mu_{\widetilde{A}}([v_0\dots v_\ell]).
	\end{eqnarray*}
	
	Consider two cylinder sets $C_1=[v_1 v_2\dots v_k]$ and $C_2=[u_1 u_2\dots u_\ell]$. As $\widetilde{A}$ is iireducible, there exists $n\geq 1$ such that $(\widetilde{A}^n)_{v_ku_1}=1$. Thus, there exists a path $w_1\dots w_n$ from $v_k$ to $u_1$ of length $n$. Hence, $[v_1 v_2\dots v_k w_1\dots w_n u_1 u_2\dots u_\ell]\subset C_1$. Thus, $\mu_{\widetilde{A}}(\sigma^{m+n}(C_1)\cap C_2)>0$. Therefore, the map $\sigma$ is ergodic.
\end{proof}

An invertible ergodic non-singular measure-preserving transformation with respect to a non-atomic measure is necessarily conservative (see~\cite[Proposition 1.2.1]{inferg}). Note that argument in the proof~\cite[Proposition 1.2.1]{inferg} rules out total dissipative case. Consequently, we have the following result. 
    
\begin{theorem}
The map $\sigma$ is conservative with respect to the measure $\mu_{\widetilde{A}}$.
\end{theorem}

Using the Poincar\'e Recurrence theorem for a conservative map~\cite{inferg}, we have that the shift map $\sigma$ is recurrent. Now we prove that the matrix $\widetilde{A}$ is recurrent, which by Theorem~\ref{thm:nonstoc-reducible-case} implies that the matrix is null recurrent.

\begin{theorem}\label{thm:null-rec}
    The matrix $\widetilde{A}$ is null recurrent.
\end{theorem}
\begin{proof}
	By Theorem~\ref{thm:nonstoc-reducible-case}, it is enough to show that $\widetilde{A}$ is recurrent.	Since $\sigma$ is a conservative map, it satisfies the Poincar\'e recurrence theorem. That is, for a measurable map $h:\widetilde{\Sigma}_{\widetilde{A}}\to \mathbb{R}$,
	\[\liminf_{n\to\infty}|h(\mathbf{x})-h(\sigma^n\mathbf{x})|=0, \text{ for a.e. } \mathbf{x}\in \widetilde{\Sigma}_{\widetilde{A}}.\]
Now consider a cylinder $B\in{\mathcal{B}_{\widetilde{A}}}$ and the indicator function $h=\chi_B$ on $\widetilde{\Sigma}_{\widetilde{A}}$, which is measurable. Hence, 
	\[\liminf_{n\to\infty}|\chi_B(\mathbf{x})-\chi_B(\sigma^n\mathbf{x})|=0, \text{ for a.e. } \mathbf{x}\in \widetilde{\Sigma}_{\widetilde{A}}.\]
Thus, for a.e. $\mathbf{x}\in B$, its orbit under $\sigma$ visits $B$ infinitely many times. Thus, $B$ is recurrent and by~\eqref{eq:Brec}, \[\mu_{\widetilde{A}}(B) = \mu_{\widetilde{A}}(B)\sum_{n\geq 1}\frac{{({\widetilde{A}}^n)_{(i,j),(i,j)}^*}}{\lambda^n},\]
	where ${({\widetilde{A}}^n)_{(i,j),(i,j)}^*}$ is the number of closed loops at $(i,j)$ in $\mathcal{G}_{\widetilde{A}}$ which return to $(i,j)$ in $n$ steps but not before that. Hence,
	\[\sum_{n\geq 1}\frac{{({\widetilde{A}}^n)_{(i,j),(i,j)}^*}}{\lambda^n}=1,\]
	which implies the matrix $\widetilde{A}$ is recurrent.
\end{proof}

Now we define an ergodic invariant $\sigma$-finite measure for $(\mathcal{L},\widetilde{f})$ as in the following theorem.

\begin{theorem}\label{thm:ergth}
The lift $\widetilde{f}$ is ergodic with respect to the  Borel $\sigma$-finite measure $\widetilde{\mu}:=\mu_{\widetilde{A}}\circ \widetilde{\pi}^{-1}$ on $\mathcal{L}$.
\end{theorem}

\begin{proof} 
We first define a $\sigma$-algebra of subsets of $\mathcal{L}$ and a measure on it. Define $\widetilde{\mathcal{B}}:=\{B\subset \mathcal{L}\ : \ \widetilde{\pi}^{-1}(B)\in \mathcal{B}_{\widetilde{A}}\}$ and $\widetilde{\mu}(B)=\mu_{\widetilde{A}}(\widetilde{\pi}^{-1}(B))$. Since $\sigma$ is measure-preserving with respect to the measure $\mu_{\widetilde{A}}$ and $\widetilde{f}\circ \widetilde{\pi}=\widetilde{\pi}\circ \sigma$, we have that the map $\widetilde{f}$ is a measure-preserving map with respect to the measure measure $\widetilde{\mu}$.

Suppose $\widetilde{f}$ is not ergodic with respect to the measure $\widetilde{\mu}$, then there exists $B\in \widetilde{\mathcal{B}}$ with $\widetilde{\mu}(B)>0$ and $\widetilde{\mu}(B^c)>0$ such that $\widetilde{f}^{-1}(B)=B$. As $\widetilde{\pi}$ is onto, $B_0=\widetilde{\pi}^{-1}(B)\ne\emptyset$ and $B_0^c=\widetilde{\pi}^{-1}(B^c)\ne\emptyset$. Moreover,  $\mu_{\widetilde{A}}(B_0)=\widetilde{\mu}(B)>0$ and $\mu_{\widetilde{A}}(B_0^c)=\widetilde{\mu}(B^c)>0$. Also, $\sigma(B_0)=\sigma(\widetilde{\pi}^{-1}(B))=\widetilde{\pi}^{-1}(\widetilde{f}^{-1}(B))=\widetilde{\pi}^{-1}(B)=B_0$. This contradicts the fact that $\sigma$ is ergodic with respect to the measure $\mu_{\widetilde{A}}$. Hence, we have obtained the desired result.
\end{proof}
 
The measure $\widetilde{\mu}$ is related to the ergodic Borel probability measure $\mu$ corresponding to the pseudo-Anosov map $f$ on $S_g$ as follows. Recall the compact subsurfaces $\{K_j\}_{j\in\mathbb{Z}}$ of the ladder surface $\mathcal{L}$. Consider the rectangle $B=R_{i,j}$ in $K_j$. Then 
\[
\widetilde{\mu}(B)=\mu_{\widetilde{A}}\circ \widetilde{\pi}^{-1}(B)=\mu_{\widetilde{A}}([(i,j)])=p_i=\mu(q_g(B)).
\] 
Using the properties of a sigma-algebra and a measure, for any measurable set $B\in \widetilde{B}$,
\[
\widetilde{\mu}(B)=\sum_{j\in\mathbb{Z}}\mu(q_g(B\cap K_j)).
\]
Note that the measure $\widetilde{\mu}$ is independent of the choice of the Markov partition on $(S_g,f)$. Finally, we have the following desired result.

\begin{theorem}\label{thm:nullreccmapth}
      The lift $\widetilde{f}$ is null recurrent.
 \end{theorem}
 
 \begin{proof}
 	Since matrix $\widetilde{A}$ is null recurrent, thus $(\widetilde{\Sigma}_{\widetilde{A}}, \sigma)$ is null recurrent. Hence, by the semi-conjugacy $\widetilde{\pi}$, the system $(\mathcal{L}, \widetilde{f})$ is also null recurrent.
 \end{proof}

Thus, since the lift $\widetilde{f}$ is a conservative map, it satisfies the Poincar\'e recurrence theorem but not Kac's lemma for the measure space $(\mathcal{L},\widetilde{\mathcal{B}},\widetilde{\mu})$. 

Since $\widetilde{f}(K_j)\cap K_j'\neq \phi $ if and only if $|j-j'|\le 1$, we get
 that \[\widetilde{\mu}(\widetilde{f}(K_j))=L_j+C_j+R_j,\]
 where $L_j=\widetilde{\mu}(\widetilde{f}(K_{j})\cap K_{j-1}))$, $C_j=\widetilde{\mu}(\widetilde{f}(K_{j})\cap K_{j}))$, and $R_j=\widetilde{\mu}(\widetilde{f}(K_{j})\cap K_{j+1}))$. Since $h_{g-1} \circ \widetilde{f} \circ h_{g-1}^{-1}=\widetilde{f}$, we have that $L_j=L$ and $R_j=R$, for all $j\in\mathbb{Z}$, for some $L>0$ and $R>0$. The quantity $R$ is the $\widetilde{\mu}$-measure of the region of $K_j$ that goes from left to right across $\partial (K_j)\cap \partial (K_{j+1})$ upon applying the map $\widetilde{f}$ and the quantity $L$ is the $\widetilde{\mu}$-measure of the region of $K_j$ that goes from right to left across $\partial (K_j)\cap \partial (K_{j-1})$ upon applying the map $\widetilde{f}$. We have the following result that follows from Theorem~\ref{thm:null-rec} (which shows that $\widetilde{A}$ is null recurrent) and~\cite[Theorem 7.3.1]{QBD} (which gives a necessary and sufficient (drift) condition for null recurrence).
 
\begin{theorem}\label{thm:left-right-drift}
We have 
\[
\widetilde{\mu}(\widetilde{f}(K_{j})\cap K_{j-1})) = \widetilde{\mu}(\widetilde{f}(K_{j})\cap K_{j+1}),
\]     
for all $j\in\mathbb{Z}$, and these quantities are nonzero and are independent of $j$. \end{theorem}

Using the above results where we establish that $\widetilde{f}$ is topologically transitive and mixing, we obtain two further results about the foliations and density of periodic points for $(\mathcal{L},\widetilde{f})$, the proof of which are similar to the analogous results for $(S_g,f)$. 
       
    \begin{cor}
       Any infinite half leaf of $(\widetilde{\mathcal{F}}_s,\widetilde{\mu}_s)$ and $(\widetilde{\mathcal{F}}_u,\widetilde{\mu}_u)$ originating from a singularity is dense on $\mathcal{L}$. 
         \end{cor}
     \begin{proof}  
     	Let $x$ be a singularity of the foliations $\widetilde{\mathcal{F}}_s$ and $\widetilde{\mathcal{F}}_u$ on $\mathcal{L}$. Consider an infinite half leaf $l$ of $\widetilde{\mathcal{F}}_u$ originating from $x$. Since $\widetilde{f}$ fixes singularities, there exists $k>0$ such that $\widetilde{f}^k(l)=l$. Let $R$ be the rectangle formed by the foliations such that $\operatorname{int}(R)\cap l \neq \emptyset$ and $x\in R$. Note that $\cap_{n>0}\widetilde{f}^{kn}(R)=l$. For any open set $U$ in $\mathcal{L}$, there exists a transverse arc $\tau\subset \widetilde{\mathcal{F}}_u$ that lies inside $U$ and has positive $\widetilde{\mu}_u$-measure. Since $\widetilde{f}$ is topologically mixing, for any transverse $\tau$ in $\mathcal{L}$, there exists $N>0$ such that for all $n\geq N$, $\widetilde{f}^{kn}(R)\cap \tau\neq \emptyset$. Hence, $l\cap U\neq \emptyset$. Thus, $l$ is dense in $\mathcal{L}$. The proof for the stable foliation is similar. 
        \begin{figure}[htbp]
        \centering
        \begin{tikzpicture}

            \node[anchor=south west, inner sep=0] (image2) at (0,0) {\includegraphics[width=0.4\linewidth]{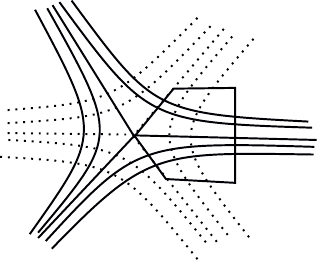}};
           
            \begin{scope}[x={(image2.south east)}, y={(image2.north west)}]
               
                \node at (0.5, 0.53) {$x$};
                 \node at (1.02, 0.48) {$l$};
                  \node at (0.78, 0.68) {$R$};
            \end{scope}
        \end{tikzpicture}
        \caption{Rectangle $R$ and infinite half leaf $l$.}
        \label{fig:dense-leaf}
    \end{figure}

\end{proof}         

 \begin{cor}
       The collection of periodic points of $\widetilde{f}$ is dense in $\mathcal{L}$. 
         \end{cor}
         \begin{proof}
         	The proof is similar to that of~\cite[Proposition 9.20]{flp} about density of the collection of periodic points of a pseudo-Anosov map on a compact surface.
         \end{proof}
      
\section{\texorpdfstring{An example of a family of null recurrent pseudo-Anosov-like maps on $\mathcal{L}$}{An example of a family of null recurrent pseudo-Anosov-like maps on L}}\label{sec:example}

We will use the generating set of $\mathrm{LMod}_{q_g}(S_g)$ as described in Theorem~\ref{thm:deygen} to construct a family of null recurrent pseudo-Anosov-like maps on the ladder surface $\mathcal{L}$ as lifts of pseudo-Anosov maps on $S_g$ under $q_g$. Following the notation in Theorem~\ref{thm:deygen}, for $1 \leq i \leq g-2$, consider the mapping class $G_i =T_{\alpha_i}^{-1}T_{i(\alpha_i)}$ and let $\beta_i = G_i(c_{i+1})$, as shown in Figure~\ref{fig:beta-curve}. Since $G_i \circ T_{c_i} \circ G_i^{-1} = T_{\beta_i}$, it follows that $T_{\beta_i} \in \mathrm{LMod}_{q_g}(S_g)$.
\begin{figure}[htbp]
        \centering
        \begin{tikzpicture}
           
            \node[anchor=south west, inner sep=0] (image2) at (0,0) {\includegraphics[width=0.5\linewidth]{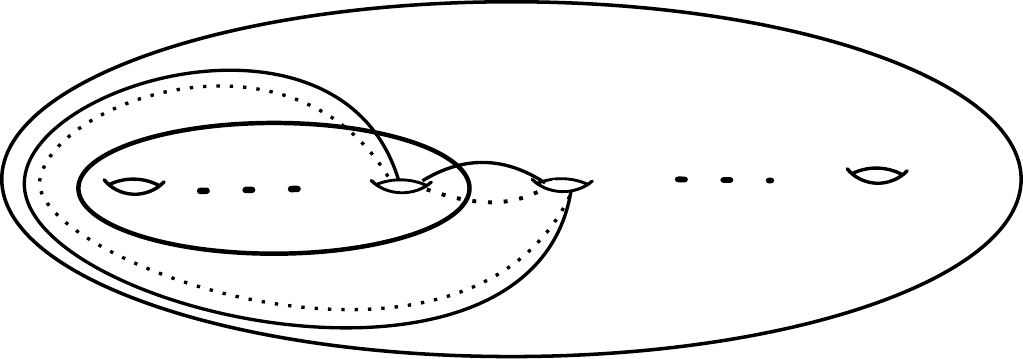}};
           
            \begin{scope}[x={(image2.south east)}, y={(image2.north west)}]
               
                \node at (0.58, 0.26) {$\beta_i$};
                 \node at (0.3, 0.55) {$\alpha_i$};
                  \node at (0.5, 0.65) {$c_{i+1}$};
            \end{scope}
        \end{tikzpicture}
        \caption{The curve $\beta_i$ in $S_g$.}
        \label{fig:beta-curve}
    \end{figure}

\begin{example}
Let $\mathcal{B}= \{\beta_1,\ldots,\beta_{g-2}\}$ and consider any nonempty subset $\mathcal{B}'$ of $\mathcal{B}$. Let $\mathcal{C}$ and $\mathcal{D}$ be multicurves on $S_g$ given by:
\[\mathcal{C}=\{c_1,\dots,c_{g-1},a_g\}\cup \mathcal{B}',  \ \  \mathcal{D}=\{b_2,\dots,b_g\}.\]
Since $S_g\setminus\mathcal{C} \cup \mathcal{D}$ is a finite union of embedded disks in $S_g$, it $\mathcal{C} \cup \mathcal{D}$ fills $S_g$. Thus, by Theorems~\ref{thm:deygen} and~\ref{thm:penner}, we obtain a family $\mathcal{F}_{\mathcal{B}'}$ of liftable pseudo-Anosov mapping classes in $\mathrm{LMod}_{q_{g}}(S_g)$ that factor into left-handed Dehn twists about the curves in $\mathcal{C}$ and right-handed Dehn twists about the curves in $\mathcal{S}$ (where each Dehn twist appears at least once). By Corollary~\ref{cor:seq}, the family $\mathcal{F}_{\mathcal{B}'}$ lifts to a family $\widetilde{\mathcal{F}}_{\mathcal{B}'}$ of null recurrent Penner-like pseudo-Anosov mapping classes in $\mathrm{Map}(\mathcal{L})$ by Theorem~\ref{thm:nullreccmapth}. Observe that one can replace the choices for $\mathcal{C}$ and $\mathcal{D}$ with \[\mathcal{C}=\{c_1,\dots,c_{g-1},a_g\}\cup B', \ \ \mathcal{D}=\{a_1,b_2,\dots,b_g\},\] or \[\mathcal{C}=\{c_1,\dots,c_{g-1}\}\cup B', \ \  \mathcal{D}=\{a_1,b_2,\dots,b_g\},\] to construct new families of such maps.

We will now describe a lift $\widetilde{F}$ of a particular $F \in \mathrm{LMod}_{q_3}(S_3)$ in the family $\mathcal{F}_{\mathcal{B}'}$. The mapping class $F = T^{-1}_{b_2} T^{-1}_{b_3} T_{c_1} T_{c_2} T_{a_3} T_{\beta_1}$ lifts to \[\widetilde{F} = \lim_{k \to \infty} \left( \prod_{i = -k}^{k} T^{-1}_{b_{i,2}} \right) \left( \prod_{i = -k}^{k} T^{-1}_{b_{i,3}} \right) \left( \prod_{i = -k}^{k} T_{c_{i,1}} \right) \left( \prod_{i = -k}^{k} T_{c_{i,2}} \right) \left( \prod_{i = -k}^{k} T_{a_{i,3}} \right) \left( \prod_{i = -k}^{k} T_{\beta_{i,1}} \right),
\] 
under the cover $q_3$, where the factor twists are about the curves shown in Figure~\ref{fig:comparison-lift-examples}.
\begin{figure}[htbp]
        \centering
         \begin{tikzpicture}
            
            \node[anchor=south west, inner sep=0] (image) at (0,0) {\includegraphics[width=0.6\linewidth]{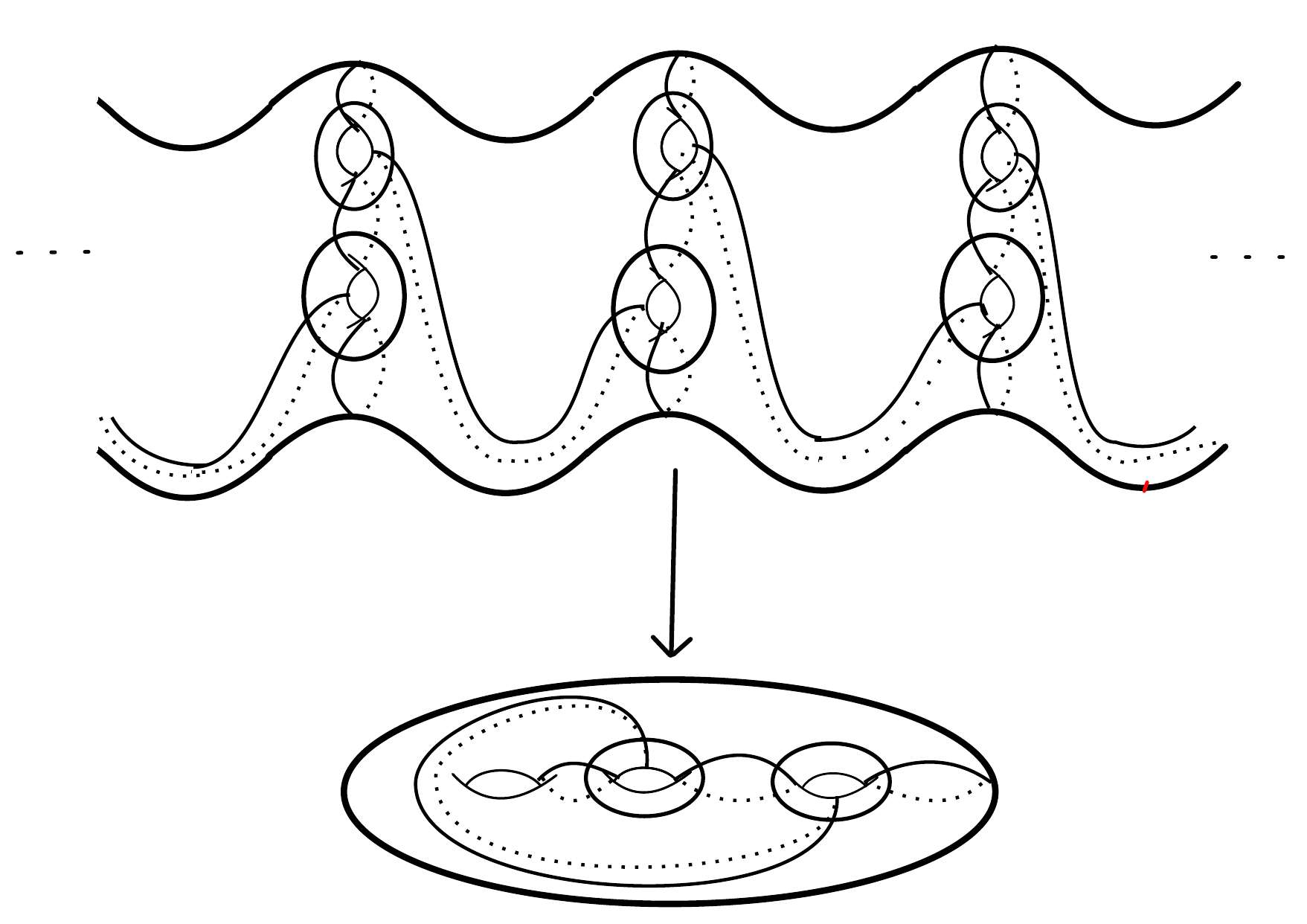}};

          \begin{scope}[x={(image.south east)}, y={(image.north west)}]
             
               \node[scale=0.8] at (0.27, 0.98) {$a_{-1,3}$};
                \node[scale=0.8] at (0.53, 0.98) {$a_{0,3}$};
                \node[scale=0.8] at (0.77, 0.98) {$a_{1,3}$};

                \node[scale=0.8] at (0.18, 0.68) {$b_{-1,2}$};
                \node[scale=0.8] at (0.2, 0.82) {$b_{-1,3}$};
                \node[scale=0.8] at (0.44, 0.68) {$b_{0,2}$};
                \node[scale=0.8] at (0.46, 0.85) {$b_{0,3}$};
                \node[scale=0.8] at (0.69, 0.7) {$b_{1,2}$};
                \node[scale=0.8] at (0.7, 0.85) {$b_{1,3}$};

                \node[scale=0.8] at (0.21, 0.76) {$c_{-1,2}$};
                \node[scale=0.8] at (0.46, 0.76) {$c_{0,2}$};
                \node[scale=0.8] at (0.71, 0.76) {$c_{1,2}$};

                \node[scale=0.8] at (0.15, 0.6) {$\beta_{-2,1}$};
                \node[scale=0.8] at (0.4, 0.6) {$\beta_{-1,1}$};
                \node[scale=0.8] at (0.65, 0.6) {$\beta_{0,1}$};
                \node[scale=0.8] at (0.882, 0.6) {$\beta_{1,1}$};

                \node[scale=0.8] at (0.27, 0.5) {$c_{-1,1}$};
                \node[scale=0.8] at (0.51, 0.51) {$c_{0,1}$};
                \node[scale=0.8] at (0.76, 0.5) {$c_{1,1}$};

                \node[scale=0.8] at (0.52, 0.22) {$b_{2}$};
                \node[scale=0.8] at (0.61, 0.22) {$b_{3}$};
                \node[scale=0.8] at (0.57, 0.2) {$c_{2}$};
                \node[scale=0.8] at (0.44, 0.2) {$c_{1}$};
                \node[scale=0.8] at (0.79, 0.15) {$a_{3}$};
                \node[scale=0.8] at (0.3, 0.15) {$\beta_{1}$};
                \node at (0.57, 0.35) {$q_3$};
            \end{scope}
        \end{tikzpicture}
        \caption{The lift of $F \in \mathrm{LMod}_{q_3}(S_3)$ to $\mathrm{Map}(\mathcal{L})$.}
        \label{fig:comparison-lift-examples}
    \end{figure}
\end{example}
\section{Acknowledgements} 
The first author is supported by the ANRF, Department of Science and Technology, India, File No. MTR/2023/000167. The second author is supported by the Prime Minister’s Research Fellowship (PMRF) (Id: 0403022), Ministry of Education, Government of India. The third author is supported by an ANRF, Department of Science and Technology, India, File No. MTR/2023/000610.

\bibliographystyle{plain} 
\bibliography{pA_Ladder}
\end{document}